\documentclass[a4paper,oneside,11pt]{article}
\usepackage{amsmath}
\usepackage{amsfonts}
\usepackage{amssymb}
\usepackage{titlesec}
\usepackage{graphicx}
\usepackage{subfigure}
\usepackage{soul}
\usepackage{fancyhdr}
\usepackage{float}
\usepackage{pdfsync}
\usepackage{latexsym}
\usepackage{mathrsfs}
\usepackage{mathtools}
\usepackage{amsthm}
\usepackage{wasysym}
\usepackage{cite}
\usepackage{color}
\usepackage{geometry}
\usepackage{extarrows}
\usepackage{enumerate}
\usepackage{dsfont}
\usepackage{bm}
\usepackage{enumitem}
\DeclareMathOperator{\esssup}{esssup}
\usepackage[marginal]{footmisc}
\usepackage[dvips]{epsfig}
\usepackage{amscd}
\usepackage[bookmarks,colorlinks,citecolor=blue,linkcolor=red]{hyperref}
\allowdisplaybreaks

\numberwithin{equation}{section}
\newtheorem{definition}{Definition}[section]
\newtheorem{theorem}{Theorem}[section]

\newtheorem{lemma}{Lemma}[section]
\newtheorem{remark}{Remark}[section]
\newtheorem{proposition}{Proposition}[section]
\renewcommand{\thefootnote}{}

\title{A Regularized Framework and Admissible Solutions for Liquid-Vapor Phase Transitions in Steady Compressible Flows }
\date{  }
\author{Yazhou Chen$^1$, Qiaolin He$^2$, Dongjuan Niu$^{3}$, Yi Peng$^1$, Xiaoding Shi$^{1*}$ \\[3mm]
\scriptsize$^{1}$ {College of Mathematics and Physics, Beijing University of
Chemical Technology, Beijing 100029, China}\\
\scriptsize$^2$ {School of Mathematical Science, Sichuan University School, Beijing 100048, China }\\
\scriptsize$^3$ {School of Mathematical Science, Capital Normal University School, Beijing 100048, China }
}

\begin{document}
\maketitle
\renewcommand{\thefootnote}{\fnsymbol{footnote}}
\footnotetext[1]{{Corresponding author. }\\
{Email: chenyz@mail.buct.edu.cn (Y.Chen), qlhejenny@scu.edu.cn (Q.He), djniu@cnu.edu.cn,(D.Niu), 2024500077@buct.edu.cn (Y.Peng), shixd@mail.buct.edu.cn (X.Shi)}}

\begin{abstract}
We investigate the well-posedness of the periodic boundary value problem for the steady compressible isentropic Navier-Stokes system under the van der Waals equation of state. The main difficulty arises from the non-monotonicity of the pressure, which induces liquid-vapor phase transitions and consequently leads to both physical instabilities and mathematical non-uniqueness of solutions. It is shown that the occurrence of a phase transition is determined by whether the integral average of the specific volume lies inside the gas-liquid coexistence region defined by the Maxwell construction. By introducing an artificial viscosity, we construct an approximate system. When the integral average of the specific volume falls within the Maxwell region, the approximate solution converges, as the artificial viscosity tends to zero, to the equilibrium states given by Maxwell’s construction, with the diffuse interface sharpening into a discontinuity. Conversely, if the integral average of the specific volume lies outside this region, the limiting solution remains outside as well, meaning that no phase transition occurs. These results demonstrate that the non-monotonicity of the pressure, combined with the condition that the integral average of the specific volume belongs to the Maxwell region, can act as a nucleation mechanism for phase transitions in the isentropic gas-liquid problem. Furthermore, the proposed approximation not only offers a regularized framework for describing phase transitions but also provides, from a rigorous mathematical viewpoint, a definition of admissible solutions related to phase transitions. The detailed proof relies on the artificial viscosity method, the calculus of variations, the anti-derivative technique, phase-plane analysis, and the level-set method.
\end{abstract}


\


\section{Introduction and Main Results}\label{sec:int}

 \ \ \ \ The ideal gas equation of state 
 $pv=R\theta$, relates pressure $p$  molar volume $v$, represents the molar and temperature $\theta$, with $R>0$ being the gas constant. This model rests on two key assumptions: gas molecules are volumeless point particles with no intermolecular forces. Consequently, all collisions are perfectly elastic, conserving kinetic energy. In contrast, the van der Waals equation introduces two corrections to account for real gas behavior: first, a co-volume term (or excluded volume) compensates for the finite size of molecules, reducing the available volume for motion. Second, a pressure reduction term (or mean-field attraction) accounts for the weak attractive forces between molecules, which lowers the observed pressure. By incorporating these molecular-scale features, the van der Waals equation, named after the 1910 Nobel Laureate in physics (Van der Waals \cite{W-1894}), offers a more accurate and physically realistic description of non-ideal gases.
To the best of our knowledge (see \cite{W-1894}, \cite{HW-1997}, \cite{HK-1993}, \cite{MLW-2007-1}, \cite{MLW-2007-2}, \cite{HS-2021} and references therein),
 the original van der Waals equation is given by
\begin{equation}\label{van der Waals equation}
 (p+\frac{a}{v^2})(v-b)=R\theta,
\end{equation}
where   $a>0$ and $b>0$ are constants characterizing the molecular cohesive force and the finite molecular size, respectively. Physically, $b$ represents the excluded volume due to the molecules themselves.

This paper examines flow dynamics in liquid-vapor phase transitions. It is known that such two-phase systems can be described by a single-fluid Navier-Stokes model, where both phases are treated as a unified continuum. Within this framework, the density field acts as an effective phase-field variable, distinguishing vapor (low-density regions) from liquid (high-density regions). Phase transitions are represented through the van der Waals equation of state \eqref{van der Waals equation}, originally proposed by van der Waals \cite{W-1894}. The motion of an isothermal compressible van der Waals fluid is governed by the following Navier-Stokes system:
\begin{equation}\label{original NS}
\left\{\begin{array}{llll}
\displaystyle \rho_{t}+\textrm{div}(\rho \mathbf{u})=0,\\
\displaystyle (\rho \mathbf{u})_{t}+\mathrm{div}\big(\rho \mathbf{u}\otimes\mathbf{u}\big)-\nu_1\Delta\mathbf{u}-\big(\nu_1+\nu_2\big)\nabla\mathrm{div}\mathbf{u}+\nabla p=0.
\end{array}\right.
\end{equation}
Here, $\rho$ denotes the mass density, $\mathbf{u}$ the flow velocity, and the pressure $p$ is related to the specific volume $v$ via the equation of state \eqref{van der Waals equation} with $\rho = 1/v$.
The operators $\nabla$, $\operatorname{div}$, and $\Delta$ represent the gradient, divergence, and Laplace operators, respectively, while the subscript $t$ indicates the partial derivative with respect to time.
The parameters $\nu_1 > 0$ and $\nu_2 > 0$ correspond to the shear and bulk viscosity coefficients. $p$ represents pressure and it satisfies the van der Waals state equation \eqref{van der Waals equation}.

To analyze the well-posedness and other aspects of the system \eqref{original NS}, we begin by examining the van der Waals equation of state \eqref{van der Waals equation} and presenting its fundamental properties.
According to the van der Waals equation \eqref{van der Waals equation}, for temperatures in the range $0 < \theta < \theta_c$, the corresponding isotherm on the $p$--$v$ diagram exhibits a non-monotonic segment characterized by a local maximum and a minimum. This ``hump'' signifies a region of mechanical instability where pressure increases with volume. This is precisely the region where liquid-vapor phase transition occurs, as illustrated in Figure 1. A straightforward calculation yields the critical temperature,
\begin{equation}\label{critical temperature}
    \theta_c = \frac{8a}{27Rb},
\end{equation}
which serves as the upper bound for such unstable behavior. Throughout this work, we operate under the isothermal assumption with $\theta < \theta_c$, a necessary condition for phase change.

\begin{lemma}\label{properties of pressure}
\textbf{(Properties of the pressure function.)} Consider a subcritical temperature $\theta$, satisfying
\begin{equation}\label{Critical temperature condition}
0 < \theta < \theta_c.
\end{equation}
Under this condition, the pressure function $p(v)$ given by \eqref{van der Waals equation} exhibits the characteristic form shown in Figure 1 below:
\begin{enumerate}[label=(\roman*)]
    \item $\displaystyle \lim_{v \to +\infty} p(v) = 0^+$;

    \item The function $p(v)$ is non-monotone and possesses exactly two critical points $\beta > \alpha > b$ such that:
    \begin{itemize}
        \item $p'(\alpha) = p'(\beta) = 0$,
        \item $\alpha$ is the point of local minimum, $\beta$ is the point of local maximum,
        \item $p(v)$ is strictly decreasing on $(b, \alpha)$ and $(\beta, +\infty)$, and strictly increasing on $(\alpha, \beta)$;
    \end{itemize}

    \item There exist unique points $\alpha_0 \in (b, \alpha)$ and $\beta_0 \in (\beta, \infty)$ such that the \textbf{Maxwell construction} (i.e. Maxwell equal area rule) \cite{M1875} holds:
    \begin{equation}\label{Maxwell equal area rule}
        \int_{\alpha_0}^{\beta_0} \big[ p(v) - p(\alpha_0) \big] \, dv = 0;
    \end{equation}
    Graphically, this corresponds to the horizontal chord $BC$ in Fig.1 for which the areas $S_{CGF}$ and $S_{GEB}$ are equal, ensuring thermodynamic equilibrium. From a physical perspective, the interval  $(\alpha_0,\beta_0)$ is known as the coexistence region, commonly referred to as the \textbf{Maxwell region}.
     
   \item There exist unique points $\bar\alpha \in (b, \alpha_0)$ and $\bar\beta \in (\beta_0, \infty)$ such that the 
      \begin{equation*}
   p(\bar\beta)=p(\alpha)\xlongequal{\mathrm{def}}\underline\sigma,\ \ p(\bar\alpha)=p(\beta)\xlongequal{\mathrm{def}}\overline\sigma;
   \end{equation*}
   
   \item Given any $\sigma \in (\underline{\sigma}, \overline{\sigma})$, the equation $p(v) = \sigma$ has precisely three distinct solutions, denoted by $\alpha_\sigma$, $\xi_\sigma$ and $\beta_\sigma$, satisfying
\[
p(\alpha_\sigma) = p(\xi_\sigma) = p(\beta_\sigma) = \sigma.
\]
\end{enumerate}
\end{lemma}
 
\

\begin{remark}\label{Maxwell}
\textbf{(Maxwell region.)}
Within the theoretical framework of thermodynamic phase transitions, the Maxwell equal-area rule \cite{M1875} prescribes the conditions for phase coexistence. Its fundamental principle is to enforce equality of pressure between two coexisting phases by requiring that the net work done in a hypothetical cycle along an isotherm vanishes. Mathematically, this is expressed by equating the area under the $p$--$v$ isotherm to the area of the rectangle under the constant-pressure tie-line, which yields the condition for $\alpha_0$ and $\beta_0$ (introduced in Lemma \ref{properties of pressure}):
\begin{equation}\label{Maxwell equal area}
\int_{\alpha_0}^{\beta_0} p(v, \theta)  dv = p(\alpha_0, \theta)(\beta_0 - \alpha_0), \qquad \text{for fixed } 0 < \theta < \theta_c.
\end{equation}
Geometrically, this implies the equality of the areas above and below the horizontal segment $BC$ (corresponding to $S_{CGF}$ and $S_{GEB}$ in Fig.~1), thereby ensuring thermodynamic consistency. Consequently, we define the \textbf{Maxwell region}, which will be used throughout this paper, as
\begin{equation}\label{M-R}
\Omega_{\mathrm{Maxwell}} =(\alpha_0,\beta_0).
\end{equation}
\end{remark}

\

\begin{figure}
\centering
\includegraphics[width=4.5in,height=3in]{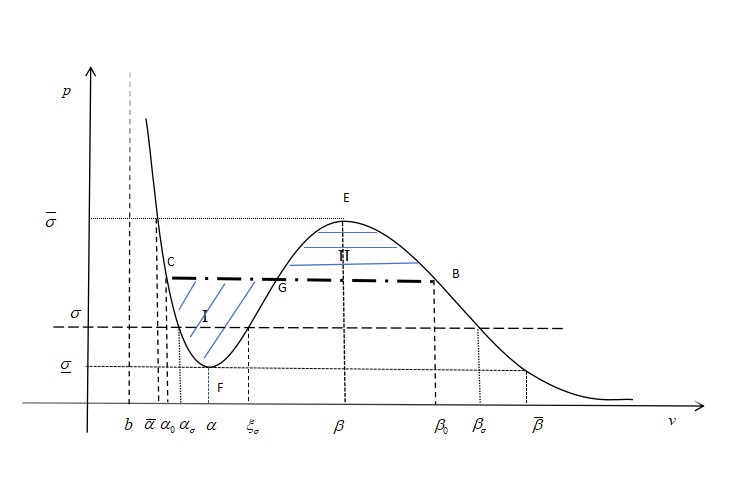}
\caption{Figure of the pressure }
\end{figure}

\

Let us now review the existing results concerning issues related to the Navier-Stokes equations \eqref{original NS}.
For the compressible isentropic (or isothermal) Navier-Stokes equations governed by the ideal gas equation of state $pv = R\theta$, problems such as the well-posedness, long-time behavior, and singular limits of their solutions have been extensively investigated. Seminal contributions include the existence of strong solutions for small initial perturbations by Matsumura-Nishida \cite{MN1980}, the framework of renormalized weak solutions by Lions  \cite{Lions1998}, and the existence theory for solutions with small energy but possibly large oscillations by Huang-Li-Xin \cite{HLX2012}. For the one-dimensional non-isentropic case, the global existence of solutions without smallness assumptions on the initial data was successively established by  Kanel \cite{K1979} and  Kazhikhov \cite{K1981}, further related developments can be found in Jiang \cite{J1999} and  Li-Liang \cite{LL2016} among others. Regarding the asymptotic behavior of solutions, seminal contributions include the stability of rarefaction and shock waves established by Matsumura-Nishihara \cite{MN1985, MN1986}, and the stability of contact discontinuities proved by Huang-Matsumura-Shi \cite{HMS-2004} and Huang-Li-Matsumura \cite{HLM2010}. Furthermore, the stability and viscosity limit of interacting shock waves were investigated by Huang-Wang-Wang-Yang \cite{HWWY2015} and Shi-Yong-Zhang \cite{SYZ2016}, respectively.

For fluids undergoing phase transitions, the pressure $p$ is typically a non-monotonic function of the density $\rho$, a characteristic of non-ideal equations of state. In the context of van der Waals fluid,  the phase transition solution was studied numerically by Affouf-Caflisch \cite{AC1991} for the Cauchy problem
of the isothermal fluid equations with a non-monotone equation of state (like that of van der Waals) and with viscosity and capillarity terms.   Hsieh-Wang \cite{HW-1997} employed a pseudo-spectral method with artificial viscosity, demonstrating numerically that the non-monotonicity of pressure induces phase transition. This study was later revisited by He-Liu-Shi \cite{HLS}, who applied a second-order TVD Runge-Kutta splitting scheme combined with a Jin-Xin relaxation scheme, a methodology also extended to the compressible Navier-Stokes/Cahn-Hilliard system and Navier-Stokes/Allen-Cahn system in He-Shi \cite{HS-2020} and \cite{HS-2021} respectively.  Eden-Milani-Nicolaenko \cite{EMN1993} applied semigroup methods to the periodic van der Waals Navier-Stokes equations with a Korteweg capillary term. They demonstrated that for sufficiently small initial data, the trivial steady-state solution is weakly stable, and established the existence of local, finite-dimensional exponential attractors. Separately, Zumbrun \cite{Z2000} analyzed the Cauchy problem for the viscous-capillary p-system and obtained the linear orbital stability of its traveling wave-front solutions. Hsiao \cite{H-1990} constructs admissible shock and wave curves for the one-dimensional inviscid Riemann problem, and then builds a wave fan, composed of shocks and rarefaction waves, that matches the given initial data.
For dynamic elastic bar theory, where the stress-deformation relation is similarly non-monotone (e.g., $p(\rho)=\rho^{-3}-\rho^{-1}$), Mei-Liu-Wong \cite{MLW-2007-1, MLW-2007-2, MWL-2008} established the existence of a global strong solution to the one-dimensional Navier-Stokes system with additional artificial viscosity. In a related setting, Hoff and Khodia \cite{HK-1993} investigated the dynamic stability of certain steady-state weak solutions on the one-dimensional whole space under small initial disturbances.

It should be noted that in the context of phase transition problems, employing the ideal gas equation of state is physically incorrect. A van der Waals-type or other non-ideal equation of state must be used to construct a correct bulk free energy function and to provide a physically realistic description of pressure for the Navier-Stokes equations. The ideal gas model is applicable only to single-phase flows, either purely gaseous or purely liquid, and must be replaced whenever phase transition is involved.

The aforementioned studies, particularly computational findings of Hsieh-Wang \cite{HW-1997} etc., demonstrate that a compressible fluid described by the van der Waals equation of state consistently gives rise to complex propagating phase boundaries. In particular, stable liquid and gas regions can coexist with unstable regions undergoing phase transition decomposition, these unstable regions correspond to the monotonically increasing segment of the van der Waals isothermal curve on the pressure, specific volume diagram. 

This paper thus focuses on steady-state solutions for such phase transitions, motivated by two fundamental questions. First, among the piecewise smooth mathematical solutions, which ones represent physically phase transition solution with two interfaces? Second, how can a systematic framework be developed to derive these phase transition solution with two interfaces? The steady-state solution under discussion in this paper is a weak solution of system \eqref{original NS}. As such, it represents a phase boundary wave, which together with shock waves, rarefaction waves, and contact discontinuities, forms one of the fundamental wave types in the Euler equations with a van der Waals equation of state.

Owing to the complexity of the general problem, we restrict our analysis to the one-dimensional full-space setting. As a simplifying assumption, a $2L$-periodic boundary condition is adopted for system \eqref{original NS}, and the viscosity coefficient $\epsilon$ is taken to be $1$. In one spatial dimension, the system simplifies to the following form:
\begin{equation}\label{1-d NS}
\left\{\begin{array}{llll}
\displaystyle \rho_{{\tilde t}}+(\rho u)_{\tilde{x}}=0,& (\tilde{x},{\tilde t})\in\mathbb{R}\times\mathbb{R}_+,\\
\displaystyle (\rho u)_{{\tilde t}}+\big(\rho u^2 +p(\rho)\big)_{\tilde{x}}= u_{\tilde{x}\tilde{x}},& (\tilde{x},{\tilde t})\in\mathbb{R}\times\mathbb{R}_+,\\
\displaystyle (\rho,u)(\tilde{x},{\tilde t})=(\rho,u)(\tilde{x}+2L,{\tilde t}), & (\tilde{x},{\tilde t})\in\mathbb{R}\times\mathbb{R}_+,\\
\displaystyle (\rho,u)\big|_{t=0}=(\rho_0,u_0)(\tilde x),& \tilde{x}\in\mathbb{R},
\end{array}\right.
\end{equation}
where $\rho$ and $u$ are unknown functions representing the fluid's density and velocity, respectively, and $\cdot_{\tilde t}=\frac{\partial}{\partial {\tilde t}}$, $\cdot_{\tilde{x}}=\frac{\partial}{\partial {\tilde x}}$. $(\rho_0,u_0)$ are periodic on $[-L,L]$.

\begin{remark}\label{linearized analysis}
The computations (see Hsieh-Wang \cite{HW-1997}, He-Liu-Shi \cite{HLS} and references therein) reveal that even after incorporating an artificial viscosity term for density into Equation~(1), the solution still exhibits violent oscillations within the Maxwell region. To elucidate the fundamental mechanism behind this behavior, we conduct a linearized analysis. Specifically, by introducing a density-related artificial viscosity term $\varepsilon \rho_{xx}$ into system \eqref{1-d NS}, we obtain the following parabolic nonlinear system:
\begin{equation}\label{system with artificial viscosity}
\left\{\begin{array}{llll}
\displaystyle \rho_{t}-(\rho u)_{x}=\varepsilon\rho_{xx},& (x,t)\in\mathbb{R}\times\mathbb{R}_+,\\
\displaystyle  u_{t}+uu_x+\frac{1}{\rho}\frac{d p}{d\rho}\rho_x= u_{xx},  & (x,t)\in\mathbb{R}\times\mathbb{R}_+,
\end{array}\right.
\end{equation}which is amenable to linear stability analysis.
Consider a constant state $(\rho,u)=(\rho_0,0)$ located in the region where the pressure is monotonically decreasing, i.e., $p'(\rho_0)<0$. Linearizing system \eqref{system with artificial viscosity} around this state yields
\begin{equation}\label{linearize Riemann Problem}
\begin{cases}
\displaystyle \rho_{t}-\rho_0 u_{x}=\varepsilon\rho_{xx},\\
\displaystyle u_{t}+\frac{1}{\rho_0}p'(\rho_0)\rho_x= u_{xx}.
\end{cases}
\end{equation}
Seeking solutions of the form $(\rho,u)=(a(t),b(t))e^{inx}$ and substituting into \eqref{linearize Riemann Problem} leads to
\begin{equation}\label{a-b}
\frac{d}{dt}\binom{a}{b}=A\binom{a}{b},\quad {\ \textrm{where}\ } A=\left(\begin{array}{cc}
     -\varepsilon|n|^2 & in\rho_0 \\
     \frac{in}{\rho}p'(\rho_0) & -|n|^2 
   \end{array}\right).
\end{equation}
The eigenvalues of $A$ are $\lambda=-(\varepsilon+1)|n|^2\pm\sqrt{\big[(\varepsilon+1)|n|^2\big]^2-\big(4\varepsilon |n|^4+p'(\rho_0)n^2\big)}$.
Instability occurs when $\operatorname{Re}(\lambda)>0$, which is possible provided
$|n|<\sqrt{\frac{-p'(\rho_0)}{\varepsilon}}$.
This inequality defines the range of wave numbers for which linear instability arises.
\end{remark}

The above linearized analysis in Remark \ref{linearized analysis} suggests that, in the isentropic setting, the non-monotonicity of pressure may be the fundamental cause of phase transition. On the other hand, given the jump discontinuities inherent in phase transition images, this example demonstrates, both computationally and analytically, that incorporating an artificial viscosity term provides the most direct approach for constructing approximate solutions to phase transitions. To this end, we begin by reformulating the aforementioned equation within a Lagrangian coordinate system. Unlike the Eulerian framework, the Lagrangian description naturally decouples density and velocity, thereby facilitating a more straightforward analysis, as detailed below.
Without loss of generality, we  assume throughout this paper that
\begin{equation}\label{H1}
  Lm_0=1.
\end{equation}
Now we can introduce the Lagrangian mass coordinates, which are defined as follows:
\begin{equation}\label{Lagrange co-ordinates}
  x=\int_0^{\tilde x}\rho(\xi,\tilde{t})d\xi,\qquad t=\tilde t.
\end{equation}
It is readily verifiable that the map $(\tilde x, \tilde{t}) \rightarrow (x, t)$ is a $C^1$-bijective map, and through direct computation, we obtain
\begin{equation}\label{L-trans}
 \frac{\partial}{\partial \tilde t}=-\rho u\frac{\partial}{\partial x}+\frac{\partial}{\partial t},\quad \frac{\partial}{\partial \tilde x}=\rho\frac{\partial}{\partial x}.
\end{equation}
Introducing the specific volume $v = 1/\rho$, system \eqref{1-d NS} transforms into the following periodic boundary value problem:
 \begin{equation}\label{ns-lagrange}
\left\{\begin{array}{llll}
\displaystyle v_{t}-u_{x}=0,& (x,t)\in\mathbb{R}\times\mathbb{R}_+,\\
\displaystyle u_{t}+p_{x}(v)=\big(\frac{u_{x}}{v}\big)_{x},  & (x,t)\in\mathbb{R}\times\mathbb{R}_+,\\
\displaystyle (v,u)(x-1,t)=(v,u)(x+1,t), & (x,t)\in\mathbb{R}\times\mathbb{R}_+,\\
\displaystyle  (v,u)(x,0)=(v_0,u_0)(x), & x\in\mathbb{R}.
\end{array}\right.
\end{equation}
A necessary condition is that the initial data be periodic:
\begin{equation*}
  (v_0,u_0)(x-1)=(v_0,u_0)(x+1),\qquad x\in\mathbb{R}.
\end{equation*}

Given that the periodic solutions $(v, u)(x, t)$ of system \eqref{ns-lagrange} in 
$\mathbb{R}$ can be interpreted as $2$-periodic extensions of their restrictions to $[-1,1]$, we hereafter restrict our analysis of \eqref{ns-lagrange} to the bounded interval $[-1,1]$. By integrating \eqref{ns-lagrange}$_1$ and  \eqref{ns-lagrange}$_2$ over the domain $[-1,1]\times[0,t]$ respectively,  and applying the periodic boundary condition specified in \eqref{ns-lagrange}$_3$, we derive the following requirement:
\begin{equation}\label{Hypothesis-1}
  \int_{-1}^{1}v(x,t)dx=\int_{-1}^{1}v_0(x)dx\xlongequal{\mathrm{def}}2\bar v.
\end{equation}

Since the steady-state solution describes the long-time asymptotic behavior of the unsteady flow, we begin by analyzing the steady state of system \eqref{ns-lagrange} under constraint \eqref{Hypothesis-1}.
It is well known that this steady-state problem reduces to the following system of ordinary differential equations:
\begin{equation}\label{steady-state ns-lagrange}
\left\{\begin{array}{llll}
\displaystyle -u_{x}=0,& x\in [-1,1],\\
\displaystyle p_{x}(v)=\big(\frac{u_{x}}{v}\big)_{x}, & x\in [-1,1], \\
\displaystyle  (v,u)(-1)=(v,u)(1),
\end{array}\right.
\end{equation}
supplemented by the integral constraint
\begin{equation}\label{constraint for v,u}
\frac{1}{2}\int_{-1}^{1}v(x)dx = \bar v,
\end{equation}
where $\bar v$ denotes the  integrated average of specific volume over the period $[-1,1]$.
Substituting \eqref{steady-state ns-lagrange}$_1$ into \eqref{steady-state ns-lagrange}$_2$ immediately yields
\begin{equation}\label{phase transition equation}
 \left\{\begin{array}{llll}
\displaystyle - p_{x}(v)=0,& x\in [-1,1], \\
\displaystyle v(-1)=v(1). \ 
\end{array}\right.
\end{equation}

\

In view of the preceding discussion, \textbf{the main objectives of this paper}, established in Theorems \ref{Existence}--\ref{thm:approximation}, are summarised as follows:

\begin{enumerate}[label=(\Roman*)]
    \item  \textbf{Well-posedness} of piecewise smooth solutions of \eqref{phase transition equation} with the constraint \eqref{constraint for v,u};
     \item \textbf{Regularised framework} with a rigorous definition of admissible solutions to system \eqref{steady-state ns-lagrange} under constraint \eqref{constraint for v,u}.
    \item \textbf{Singular-limit behaviour} of an artificial-viscosity approximation and the role of the average specific volume $\bar v$ as defined in \eqref{constraint for v,u};  
    \item \textbf{Nucleation mechanism} and interface formation in isentropic liquid--vapor phase transitions;
   
\end{enumerate}

\

Focusing on the aforementioned objectives (I)-(IV), we now present our main theorems and the definition of admissible solutions for system \eqref{phase transition equation}.
Following Hsieh-Wang \cite{HW-1997}, we adopt the Maxwell equal area rule (see Remark \ref{Maxwell}) to partition the domain of $v\in(b, +\infty)$. While their work classifies the regions from a computational perspective into stable, semi-stable (metastable), and unstable parts, our analysis will more clearly demonstrate the inevitability of this classification. Accordingly, for a fixed $\bar{\theta} \in (0, \theta_c)$, we divide the interval $(b, +\infty)$ into two corresponding regions: $\Omega_{\text{unstable}}$, $\Omega_{\mathrm{metastable}}$ and $\Omega_{\text{stable}}$, referred to as the unstable, metastable and stable regions, respectively, as illustrated in Fig. 1.
\begin{equation}\label{ABC}
\left.\begin{array}{llll}
\displaystyle \Omega_{\mathrm{unstable}}\xlongequal{\mathrm{def}}(\alpha,\beta), & \displaystyle \mathrm{unstable \ region},\\
\displaystyle \Omega_{\mathrm{metastable}}\xlongequal{\mathrm{def}}(\alpha_0,\alpha]\cup[\beta,\beta_0),& \mathrm{metastable \ region},\\
\displaystyle  \Omega_{\mathrm{stable}}\xlongequal{\mathrm{def}}(b,\alpha_0]\cup[\beta_0,+\infty),& \displaystyle \mathrm{stable\ region}.
\end{array}\right.
\end{equation}

Based on the definition \eqref{ABC} and in light of Remark \ref{Maxwell}, it immediately follows that the Maxwell region is exactly the union of the unstable and semi-stable regions, i.e.,
\begin{equation}\label{Maxwell region and unstable metastable}
  \Omega_{\mathrm{Maxwell}}=(\alpha_0,\beta_0)=\Omega_{\mathrm{unstable}}\cup \Omega_{\mathrm{metastable}}.
\end{equation}
With regard to the weak solutions of system \eqref{phase transition equation}, the following result concerning the existence of piecewise smooth solutions  is established.

\

\begin{theorem}\label{Existence}
System \eqref{phase transition equation} admits a unique solution if and only if the  integrated average of specific volume $\bar v$ lies in the stable region $\Omega_{\mathrm{stable}}$; this solution is precisely the constant function $\bar v$ defined by \eqref{constraint for v,u}.
Conversely, multiple phase transition solutions exist if and only if  $\bar v$ belongs to the Maxwell region $\Omega_{\mathrm{maxwell}}$. In this case, the solutions are piecewise constant functions taking only the two values $\alpha_0$ and $\beta_0$, as prescribed by the Maxwell equal-area rule \eqref{Maxwell equal area rule}.
\end{theorem}

\

\begin{remark}
Theorem \ref{Existence} demonstrates how the  integrated average of specific volume directly determines whether the steady flow of a van der Waals fluid, governed by system \eqref{phase transition equation}, exhibits phase transition solutions. Meanwhile, these phase transition solutions correspond to the minimizers of the energy functional constructed in Section \ref{sec 2} (see \eqref{Lagrange multiplier}). Mathematically, however, such an approach still yields an infinite set of possible phase transition solutions. To identify which among them is physically meaningful, further physical constraints must be introduced, naturally motivating the definition of admissible solutions (see  Definition \ref{Admissible solution}). To resolve this issue, we adopt the method of introducing artificial viscosity to construct approximate solutions (see Theorem \ref{existence of trivial solution}).
\end{remark}

Motivated by both computational and analytical considerations (see Hsieh-Wang \cite{HW-1997}, He-Liu-Shi \cite{HLS} and references therein), it is essential to develop a set of techniques for constructing smooth approximations to the discontinuous solutions described above. To this end, we adopt an approach based on an approximating system of equations.

We propose an approximate system designed to capture the relevant energy-minimizing configurations. Our approach employs viscous regularization: we analyze system \eqref{phase transition equation} under constraints \eqref{constraint for v,u} by introducing an artificial viscosity term. Specifically, adding a term of the form $\varepsilon v v_{xx}$ to the first equation of \eqref{steady-state ns-lagrange} yields the following approximate system:

\begin{equation}\label{Approximate ns-lagrange}
\left\{\begin{array}{llll}
\displaystyle -u_{x}=\varepsilon^2 vv_{xx},& x\in[-1,1],\\
\displaystyle p_{x}(v)=\big(\frac{u_{x}}{v}\big)_{x},& x\in[-1,1], \\
\displaystyle (v,u)\big|_{x=-1}=(v,u)\big|_{x=1},
\end{array}\right.
\end{equation}with constraint \eqref{constraint for v,u}. 
We now consider the periodic boundary value problem for $(v, u)$ defined by the approximate system \eqref{Approximate ns-lagrange}. The following analysis presents the corresponding existence theory.

\begin{theorem}\label{existence of trivial solution}
For the steady-state system \eqref{Approximate ns-lagrange} with artificial viscosity term $\varepsilon^2 vv_{xx}$ under constraint \eqref{constraint for v,u}, the solution is unique and trivial (i.e., $v(x) \equiv \bar v$), provided that the viscosity coefficient $\varepsilon$ satisfies
\begin{equation}\label{uniqueness conition}
\varepsilon^2\pi^2 > \max_{v \in [\alpha, \beta]} p'(v),
\end{equation}
with $\alpha$ and $\beta$ being the two critical points of pressure $p$ (Lemma \ref{properties of pressure}; see also Remark \ref{Maxwell} and Figure 1). Provided further that $\varepsilon > 0$ is sufficiently small and $\bar v \in \Omega_{\mathrm{unstable}}=(\alpha,\beta)$, at least one non-trivial solution exists.
\end{theorem}

\

\begin{remark}
Theorem \ref{existence of trivial solution} establishes the well-posedness of smooth solutions to system \eqref{Approximate ns-lagrange} with artificial viscosity $\varepsilon$ added. The theorem indicates that when $\varepsilon$ is sufficiently large, the smooth solution of \eqref{Approximate ns-lagrange} is unique regardless of the value of the volume-averaged integral. This observation is consistent with the analysis of the linearized system \eqref{system with artificial viscosity} in Remark \ref{linearized analysis}, where excessive artificial viscosity drives the wave numbers of oscillations to zero. Conversely, if the volume-averaged integral falls within the unstable region $\Omega_{unstable}$, the mathematical uniqueness of smooth solutions to \eqref{Approximate ns-lagrange} breaks down as $\varepsilon$ decreases. From a physical viewpoint, however, distinguishing which solutions are physically meaningful becomes a pressing issue. Consequently, Definition \ref{Admissible solution} and Theorem \ref{thm:approximation} emerge naturally from the discussion above as follows. It should also be particularly noted here that, the conclusion regarding the non-uniqueness of solutions can in fact be extended to the case where the volume-averaged integral lies within the Maxwell region $(\alpha_0,\beta_0)$ (see Theorem \ref{thm:approximation} and the accompanying Remark \ref{rem:6}).
\end{remark}

\


In the inviscid limit $\varepsilon \to 0$, interfaces can emerge as density jumps without an associated energy penalty. Within an energy framework that accounts for interfacial contributions, it is therefore expected that the minimal-energy two-phase periodic solution will exhibit a double-interface structure. This expectation is confirmed by numerical studies of the one-dimensional periodic problem \eqref{steady-state ns-lagrange} for a van der Waals fluid (see Hsieh--Wang \cite{HW-1997}, He--Liu--Shi \cite{HLS} and references therein). When the average initial density lies in the Maxwell region, the solution evolves from initial oscillations between the two Maxwell points toward a stable two-phase state characterized by exactly two distinct interfaces. The resulting profile, illustrated in Fig.~2, is either single-peak or single-valley; each profile corresponds to a piecewise smooth solution of system \eqref{steady-state ns-lagrange}. Guided by these combined physical and numerical insights, we introduce the concept of a \textbf{phase transition solution with two distinct interfaces}.

\

\begin{definition}\label{phase transition solution with two distinct interfaces}
A function whose profile is either single-peaked or single-dipped and which constitutes a piecewise smooth solution of system \eqref{steady-state ns-lagrange} is called a \textbf{phase transition solution with two distinct interfaces}. It is denoted by $\mathcal{V}_{\mathrm{single\text{-}peak}}(x)$ or $\mathcal{V}_{\mathrm{single\text{-}valley}}(x)$, and defined explicitly as follows (see Figure 2):
\begin{equation}\label{Phase-transition-solution-1}
  \mathcal{V}_{\mathrm{single\text{-}peak}}(x)=\begin{cases}
  \alpha_0, & x\in[-1,-1+l_{11}),\\
  \beta_0, & x\in[-1+l_{11},-1+l_{11}+l_2],\\
  \alpha_0, & x\in(-1+l_{11}+l_2, 1],
\end{cases}
\end{equation}
or
\begin{equation}\label{Phase-transition-solution-2}
  \mathcal{V}_{\mathrm{single\text{-}valley}}(x)=\begin{cases}
  \beta_0, & -1\leq x\leq-1+l_{21},\\
  \alpha_0, & -1+l_{21}<x\leq-1+l_{21}+l_1,\\
  \beta_0, & -1+l_{21}+l_1<x\leq 1,
\end{cases}
\end{equation}
where the lengths $l_1$ and $l_2$ are given by
\begin{equation}\label{lii}
  l_1 = l_{11}+l_{12}=\frac{2(\beta_0-\bar{v})}{\beta_0-\alpha_0},\qquad 
  l_2 = l_{21}+l_{22}=\frac{2(\bar{v}-\alpha_0)}{\beta_0-\alpha_0}.
\end{equation}
\end{definition}
 
 \
\begin{remark}
While the measures of the two-phase regions $l_1$ and $l_2$ can be determined, the precise locations $l_{11}$ and $l_{21}$ at which the phase transition occurs cannot be identified with the techniques employed in this work (see theorems \ref{Existence}--\ref{thm:approximation}). Hence, although the extent of the gas-liquid coexistence is established, the exact transition points remain unresolved.
\end{remark}
 
 \begin{figure}[!htbp]
 \centering\subfigure[a piecewise smooth solution with a single-peak shape $\mathcal{V}_{\mathrm{single-peak}}(x)$]
 {\begin{minipage}[t]{0.45\textwidth}\centering\includegraphics[width=2.0 in,height=2.0 in]{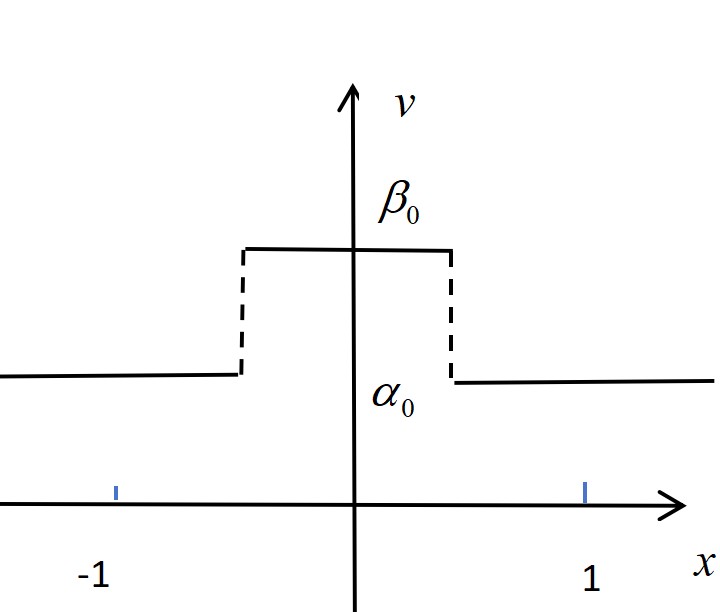}
\end{minipage}}
\hfill
\subfigure[a piecewise smooth solution with a single-valley shape $\mathcal{V}_{\mathrm{single-valley}}(x)$ ]
{\begin{minipage}[t]{0.45\textwidth}
\centering\includegraphics[width=2.5 in,height=2.3 in]{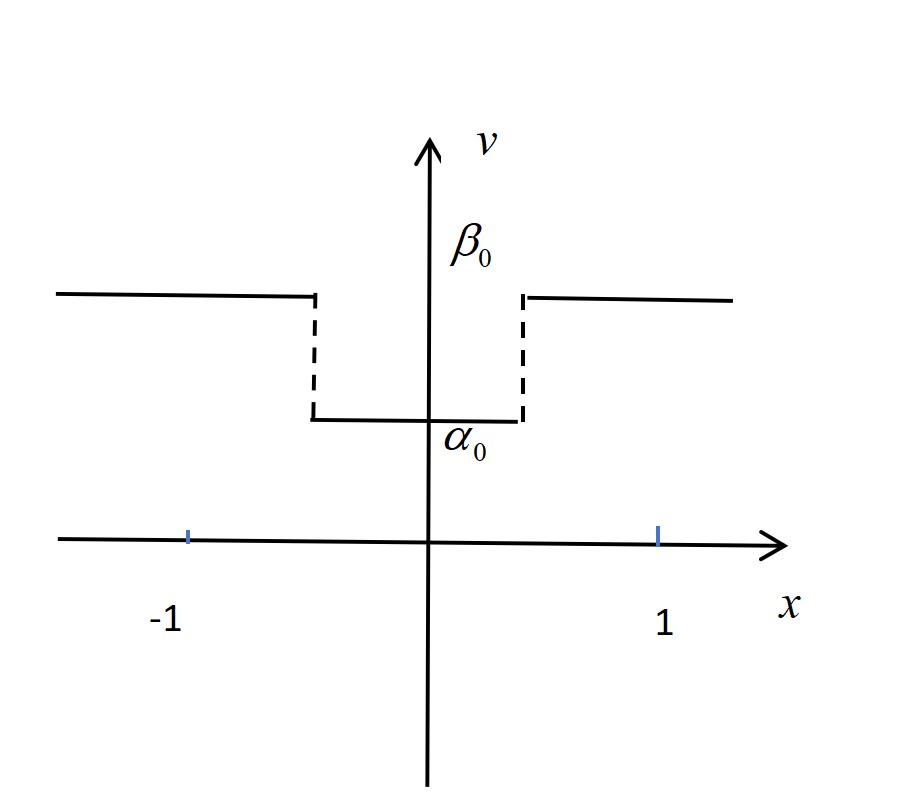}
\end{minipage}}
\centering\caption{phase transition solution with two distinct interfaces}
\label{model2figure1}\end{figure}

The introduction of Definition \ref{phase transition solution with two distinct interfaces} naturally leads us to analyze the system \eqref{phase transition equation} within a variational framework, and we proceed to formulate the minimization of the following functional
\begin{equation}\label{P0}
 \min_{v\in H_{\mathrm{per}}^1([-1,1])} \int_{-1}^1 W(v)dx,\quad \mathrm{subject\ to\ \ } \frac1{2}\int_{-1}^{1}v(x)dx=\bar v,
\end{equation}
where
\begin{equation}\label{W}
  W(v)\xlongequal{\mathrm{def}}\displaystyle-\int_{\bar v}^{v}p(s)ds.
\end{equation}

Theorem \ref{thm:approximation} below describes the convergence of physically relevant approximate solutions of the  system \eqref{Approximate ns-lagrange} to the piecewise smooth solution of the system \eqref{phase transition equation} as the artificial viscosity $\varepsilon^2$ approaches zero. The analysis relies primarily on phase-plane method, variational method,  level-set technique, Taylor expansion and the theorem of existence of inverse function. Due to the non-uniqueness of solutions to the  system \eqref{Approximate ns-lagrange}, most of which are mathematically valid but not necessarily physical, we employ the second variation method to identify those with physical significance. For clarity, we refer to this selected class of solutions, characterized by their distinct geometric and physical properties, as phase transition solutions with two distinct interfaces. The following defines a smooth approximate solution that captures the structure of phase transitions with two distinct interfaces.
 
 \
 
 \begin{definition}\label{two-smooth-interface-phase-transition-solution}
A nonconstant solution $v(x)$ of \eqref{Approximate ns-lagrange} subject to the constraint \eqref{constraint for v,u} is called a \textbf{two-smooth-interface phase transition solution} if it satisfies the following:
\begin{enumerate}
\item $v$ changes its monotonicity exactly once in the interval $[-1, 1]$;
\item $v$ is either increasing then decreasing, or decreasing then increasing.
\end{enumerate}
According to their geometric shape, such solutions are referred to as \textbf{single-peak} or \textbf{single-valley} solutions, and are respectively denoted by $v^{\varepsilon}_{\text{single-peak}}$ and $v^{\varepsilon}_{\text{single-valley}}$.
\end{definition}

\

More generally, we can define the $N$-th phase transition solution as follows:

\begin{definition}\label{2N-smooth-interface-phase-transition-solution}
A nonconstant solution $v(x)$ of \eqref{Approximate ns-lagrange} under the constraint \eqref{constraint for v,u} is called a \textbf{$2N$-smooth-interface phase transition solution} (with $N\geq 2$) if $v$ changes its monotonicity exactly $2N$ times over the interval $[-1, 1]$.
\end{definition}

 \
 
 With the preliminaries in place, we now formally define the admissible solution, that is, the physically meaningful solution, to the steady-state system \eqref{steady-state ns-lagrange}.
 
 \
 
\begin{definition}\label{Admissible solution}
(\textbf{Admissible solution})
Let ${(v^\varepsilon_{\text{single-peak}}, u^\varepsilon)}_{\varepsilon>0}$ (or ${(v^\varepsilon_{\text{single-valley}}, u^\varepsilon)}_{\varepsilon>0}$) be a family of classical solutions to the approximate system \eqref{Approximate ns-lagrange} on $\mathbb{R}$. Suppose this family is locally uniformly bounded and that each component $v^\varepsilon_{\text{single-peak}}$ (or $v^\varepsilon_{\text{single-valley}}$) is a two-smooth-interface phase transition solution (i.e., a single-peak or single-valley solution) in the sense of Definition \ref{two-smooth-interface-phase-transition-solution}. If ${(v^\varepsilon_{\text{single-peak}}, u^\varepsilon)}_{\varepsilon>0}$ (or ${(v^\varepsilon_{\text{single-valley}}, u^\varepsilon)}_{\varepsilon>0}$) converges almost everywhere to a pair $(v, u)$ as $\varepsilon \to 0^+$, then the limit $(v, u)$ is called an \textbf{admissible solution} of the steady-state system \eqref{steady-state ns-lagrange}.
\end{definition}

\

\begin{theorem}\label{thm:approximation}
    There exists $\varepsilon_0 > 0$ such that for every $0 < \varepsilon < \varepsilon_0$ and for any $\bar v$ belonging to a closed subset of $\Omega_{\mathrm{Maxwell}}=(\alpha_0, \beta_0)$, the system \eqref{Approximate ns-lagrange} under the constraint \eqref{constraint for v,u} admits exactly \textbf{two-smooth-interface phase transition solutions} possessing two interfaces: one of single-valley type, denoted by $v^{\varepsilon}_{\mathrm{single-valley}}$, and one of single-peak type, denoted by $v^{\varepsilon}_{\mathrm{single-peak}}$.
    Moreover, for each $\bar v \in (\alpha_0, \beta_0)$ and for all $x \in [-1,1]$, the following pointwise convergence holds as $\varepsilon \to 0$:
    \begin{align*}
        \lim_{\varepsilon \to 0} v^{\varepsilon}_{\mathrm{single-peak}}(x) &= \mathcal{V}_{\mathrm{single-peak}}(x), \quad x \neq -1 + l_1,\; x \neq -1 + l_1 + l_2, x\in[-1,1],\\[4pt]
        \lim_{\varepsilon \to 0} v^{\varepsilon}_{\mathrm{single-valley}}(x) &= \mathcal{V}_{\mathrm{single-valley}}(x), \quad x \neq -1 + l_2,\; x \neq -1 + l_2 + l_1, x\in[-1,1],
    \end{align*}
    where $\mathcal{V}_{\mathrm{single-peak}}$ and $\mathcal{V}_{\mathrm{single-valley}}$ denote the phase transition solutions with two distinct interfaces for system \eqref{phase transition equation} subject to constraint \eqref{constraint for v,u}, as defined in Definition~\ref{phase transition solution with two distinct interfaces}. This implies that  $\mathcal{V}_{\mathrm{single-peak}}(x)$ or $\mathcal{V}_{\mathrm{single-valley}}(x)$ are the unique admissible solutions to the steady-state system \eqref{steady-state ns-lagrange} under the constraint \eqref{constraint for v,u}.
\end{theorem}


\

\begin{remark}\label{rem:6}
Theorem \ref{thm:approximation} shows that even if the volume-averaged integral falls within the semi-stable region, the smooth solutions of system \eqref{Approximate ns-lagrange} still exhibit non-uniqueness. In fact, there exists at least one solution with two jumps, and even solutions with $2N$ jumps may arise , where the number 
$N$ is determined by the length of the period and the volume-averaged integral (see \eqref{The necessary condition for the number of phase transitions} for details). This result extends and complements Theorem \ref{existence of trivial solution}.
\end{remark}

\begin{remark}
Theorems \ref{Existence}--\ref{thm:approximation} together establish that for $\bar v \in \Omega_{\mathrm{Maxwell}}=(\alpha_0, \beta_0)$, the periodic problem for system \eqref{steady-state ns-lagrange} admits a unique two-phase steady-state solution within the class of two-sharp-interface phase transition solutions. This solution exhibits a double-interface structure with two distinct density jumps and can be constructively approximated via the artificial viscosity method. Specifically, the two-smooth-interface phase transition solutions of the regularized system converge in the $L^\infty$ norm to the corresponding two-sharp-interface solution. Conversely, if $\bar v$ lies outside the Maxwell region, the solution remains in a single-phase state.
These results indicate that the instability associated with the Maxwell region can act as a driving mechanism for phase separation. More concretely, this work identifies a nucleation mechanism triggered by pressure-induced instability: a nonmonotonic pressure profile locally perturbs the system, forming critical nuclei from which distinct phases emerge and grow. The analysis thus provides a theoretical foundation for the numerical phenomena reported in Hsieh-Wang \cite{HW-1997} and He-Liu-Shi \cite{HLS}.
\end{remark}

\
At this point, all main theorems of this paper have been stated. Their rigorous proofs will be given in Section \ref{sec 2}. Before proceeding with the detailed analysis, we provide a brief overview of the structure of Section \ref{sec 2},  the main methods used in the proofs, and the key notations adopted throughout.

The proofs are organized as follows. In Section \ref{Proof-2-1}, we use a variational approach to construct a piecewise smooth phase-transition solution for system \eqref{steady-state ns-lagrange} under constraint \eqref{constraint for v,u}. Then, in Section \ref{Proof-2-2}, we turn to the system augmented with an artificial viscosity term, namely system \eqref{Approximate ns-lagrange}, and develop a corresponding variational framework to analyze the multiplicity of its steady states. Finally, the analysis in Section \ref{Proof-2-3} yields a family of two-smooth-interface phase-transition solutions, denoted by $\{v^{\varepsilon}_{\text{single-peak}}\}_{\varepsilon>0}$ or $\{v^{\varepsilon}_{\text{single-valley}}\}_{\varepsilon>0}$. A key result establishes that these solutions constitute the unique energy minimizer among all steady states of the system. Moreover, we show that they can be obtained as the singular limit of a suitably constructed approximate model. The detailed arguments rely on the calculus of variations, phase plane analysis, and the level-set technique, which are presented comprehensively across these subsections.

\

\noindent\textbf{\normalsize Notations.} Throughout this paper, we adopt the following notations for functional spaces:
$L^\infty(\mathbb{R})$  denotes the space of all bounded measurable functions on $\mathbb{R}$, equipped with the essential supremum norm $\|f\|_{L^\infty}=$ $\mathop{\esssup}_{x\in\mathbb{R}}|f|$. For $q\geq1$, $L^q(\mathbb{R})$ (abbreviated as $L^q$ when no ambiguity arises)  denotes the space of 
$q$-integrable real-valued functions on $\mathbb{R}$,  with the norm $\|f\|_{L^q}=(\int_{\mathbb{R}}|f|^qdx)^{\frac{1}{q}}$. For the special case  $q=2$, we simplify  $\|f\|_{L^2}$ to $\|f\|$. For $l>0$, $H^l(\mathbb{R})$ (abbreviated as $H^l$ when clear from context) denotes the Sobolev space of  $L^2$-functions $f$ on $\mathbb{R}$ whose weak derivatives $\partial^j_x f,  j=1,\cdots,l$ are also square-integrable. The norm on $H^l$ is defined as$ \|f\|_{l}=(\sum_{j=0}^l\|{\partial^j_x f}\|^2)^{\frac{1}{2}}$, where  $\|\cdot\|$ is the $L^2$-norm defined above.  Additionally, we denote the periodic Sobolev space $H^l_{\mathrm{per}}([-L,L])$ with period $L$ as
\begin{equation}\label{periodic Sobolev space}
 H^l_{\mathrm{per}}([-L,L])\xlongequal{\mathrm{def}}\Big\{f(x)\in  H^l(\mathbb{R})\Big|f(x)=f(x+L),\forall x\in\mathbb{R}\Big\},\quad l\geq0,
\end{equation}
and the periodic Sobolev space $H^1_{\mathrm{per}}(\mathbb{R})$ with zero periodic integral average value
\begin{equation}\label{periodic Sobolev space}
 H^l_{\mathrm{per},0}([-L,L])\xlongequal{\mathrm{def}}\Big\{f(x)\in  H^l_{\mathrm{per}}([-L,L])\Big|\int_{-L}^Lf(x)dx=0\Big\},\quad l\geq0.
\end{equation}

\



\

\section{Proof of the Main Results}\label{sec 2}

\ \ \ \ This section details the proof of our main results. We begin by examining the multiplicity of piecewise smooth solutions to system \eqref{phase transition equation}.

\subsection{Proof of Theorem 1.1.}\label{Proof-2-1}
\ \ \ \ Due to the non-monotonic behavior of pressure, the inviscid flow governed by system \eqref{ns-lagrange} without $(\frac{u_x}{v})_{x}$  results in a set of partial differential equations of mixed hyperbolic-elliptic type. Such a mixed-type system is mathematically ill-posed and physically unstable. Consequently, the steady-state solutions of system \eqref{ns-lagrange} are not unique; multiple solutions arise, with the piecewise smooth solution being the most natural among them. This multiplicity is determined by the average value $\bar v$, as established in Theorem \ref{Existence}. The detailed proof of Theorem \ref{Existence} is provided below.

\

\begin{proof}
Applying the variational principle, we deduce that system \eqref{phase transition equation} is mathematically equivalent to the variational problem \eqref{P0}. Moreover, 
\begin{equation}\label{W-vdw}
  W(v)=-\int_{\bar v}^{v}p(s)ds=-\big(\frac{a}{v}-\frac{a}{\bar{v}}\big)-R\theta\ln\big(\frac{v-b}{\bar v-b}\big) ,\quad v>b,  
\end{equation}
and
\begin{equation}\label{derivative of W}
  W'(v)=-p(v),\qquad W''(v)=-p'(v).
\end{equation}
Using the Lagrange multiplier method, problem \eqref{P0} is transformed into minimizing  
\begin{equation}\label{Lagrange multiplier}
  \min_{v\in H_{\mathrm{per}}^1([-1,1])} \int_{-1}^1\Big(W(v(x))+\sigma v(x)\Big)dx,
\end{equation}
where $\sigma$, satisfying \eqref{constraint for sigma}, acts as a Lagrange multiplier. A necessary condition for the existence of a minimizer is that the Euler-Lagrange equation and the Weierstrass-–Erdmann corner conditions hold:
  \begin{equation}\label{conition 1}
  \left\{\begin{array}{llll}
 W'(v)=-\sigma\quad \mathrm{at\ points\ of\ continuity\ of\ } v,\ \mathrm{while}\\
 W(v)+ \sigma v\quad \mathrm{is\ continuous\ across\ jumps\ in\ } v.\end{array}\right.
 \end{equation}
From \eqref{conition 1}, it follows that solutions are either constant (single-phase) or piecewise constant (two-phase). In the latter case, they take the form
\begin{equation}\label{Piecewise smooth solution}
  v(x)=\displaystyle\left\{\begin{array}{llll}
\displaystyle \alpha_0,&\displaystyle x\in S_1,\\
\displaystyle \beta_0,&\displaystyle x\in S_2.
 \end{array}\right.
\end{equation}
where $S_1$ and $S_2$ are disjoint measurable sets with $S_1 \cup S_2 = [-1, 1]$, and $\alpha_0$, $\beta_0$ are defined by the Maxwell conditions (see Figure 1):
\begin{equation}\label{Maxwell condition}
\left\{\begin{array}{llll}
\displaystyle
  W(\beta_0)-W(\alpha_0)=-\sigma_0(\beta_0-\alpha_0),\\
  \sigma_0=-W'(\beta_0)=p(\beta_0)=-W'(\alpha_0)=p(\alpha_0).\end{array}\right.
\end{equation}
Letting
\begin{equation}\label{l}
  l_i=\mathrm{measure} (S_i),
\end{equation}
from \eqref{P0} we obtain
\begin{equation}\label{li}
 l_1=\frac{2(\beta_0-\bar v)}{\beta_0-\alpha_0},\ \ l_2=\frac{2(\bar v- \alpha_0)}{\beta_0-\alpha_0},
\end{equation}
where $\bar v$ is given in \eqref{constraint for v,u}. Since $l_1 > 0$ and $l_2 > 0$, a necessary condition for a two-phase solution is
\begin{equation}\label{phase transition condition}
  \bar v\in\Omega_{\mathrm{Maxwell}}=(\alpha_0,\beta_0).
\end{equation}
When \eqref{phase transition condition} holds, any $v(x)$ of the form \eqref{Piecewise smooth solution} with $l_i$ given by \eqref{li} is a global minimizer of \eqref{P0}. If $\bar v \in\Omega_{\textrm{stable}}=(b, \alpha_0]\cup[\beta_0,+\infty)$, a two-phase solution is impossible, and the minimizer is the single-phase solution
\begin{equation}\label{single-phase}
  v(x)=\bar v.
\end{equation}
Therefore, the solutions of \eqref{steady-state ns-lagrange} are either constant (single-phase) or piecewise constant (phase transition). These two classes of solutions, especially the latter, mathematically demonstrate the multiplicity of piecewise smooth solutions to system \eqref{steady-state ns-lagrange}.
\end{proof}

\

The proof of Theorem \ref{Existence} relies primarily on classical variational methods. This result shows that the multiplicity of solutions to system \eqref{steady-state ns-lagrange} depends on whether the average value $\bar v$ lies within the Maxwell region $(\alpha_0, \beta_0)$. Mathematically, this leads to a loss of uniqueness. Physically, however, not all such solutions are necessarily meaningful. Therefore, it becomes essential to identify the physically relevant solutions according to appropriate criteria, here the energy minimization principle introduced in the following section.

\subsection{Proof of Theorem 1.2.}\label{Proof-2-2}

\ \ \ \ We now turn to the multi-solvability of the system \eqref{Approximate ns-lagrange} with artificial viscosity, detailing the structure of its solutions and providing a proof of Theorem \ref{existence of trivial solution}.
Inspired by the works Hsieh-Wang \cite{HW-1997}, He-Liu-Shi \cite{HLS}, from the numerical analysis results of \eqref{1-d NS}, we can observe that the magnitude of viscosity affects the multiplicity of solutions for the system \eqref{Approximate ns-lagrange}. Therefore, setting
\begin{equation}\label{V}
  V=v-\bar v.
\end{equation}

Below we apply the anti-derivative method. By substituting \eqref{Approximate ns-lagrange}$_1$ into \eqref{Approximate ns-lagrange}$_2$ and integrating over $[-1, x]$, we obtain the following second-order ordinary differential equation with periodic boundary conditions for $V$:
\begin{equation}\label{second-order ODE}
\left\{\begin{array}{llll}
\displaystyle -\varepsilon^2 V_{xx}=\big(p(V+\bar v)-p(\bar v)\big)-\sigma,& x\in[-1,1], \\
\displaystyle  V(x-1)=V(x+1), \ \ V_x(x-1)=V_x(x+1), & x\in[-1,1],
\end{array}\right.
\end{equation}
subject to the constraint
\begin{equation}\label{constraint for V}
  \int_{-1}^1V(x)dx=0,
\end{equation}
where $\sigma$ is a constant given by
 \begin{equation}\label{constraint for sigma}   
 \sigma=\frac12\int_{-1}^1\Big(p\big(V(x)+\bar v\big)-p(\bar v)\Big)dx.\end{equation}

To study the existence of steady-state solutions to system \eqref{second-order ODE}, we reformulate it as a variational problem. The solution is subsequently derived by seeking a critical point of a suitably constructed functional. Specifically, we define the following functional on the space $H^1_{\mathrm{per},0}([-1,1])$:
\begin{equation}\label{G}
  G(V)\xlongequal{\mathrm{def}}\int_{-1}^1\Big(H(V)+\frac{\varepsilon^2}2V_{x}^2\Big)dx,
\end{equation}
where 
\begin{equation}\label{H}
\left.\begin{array}{llll}
 \displaystyle H(V)&=&\displaystyle\int_{0}^V\big(p(\bar v)-p(s+\bar v)\big)ds\\
 &=&\displaystyle p(\bar v)V-\int_{0}^{V}p(s+\bar v)ds\\
 &=& \displaystyle p(\bar v)V-a\big(\frac{1}{(V+\bar v)}-\frac{1}{{\bar v}}\big)-R\theta\ln\frac{V+\bar v-b}{\bar v-b}.
   \end{array}\right.
\end{equation}
The steady-state solutions is characterized by the minimization of \eqref{G}. Consequently, the corresponding minimization problem is formulated as follows:
\begin{equation}\label{Pe}
\min_{V \in H_{\mathrm{per},0}^1([-1,1])} G(V).
\end{equation}
The proof of Theorem \ref{existence of trivial solution} relies on the following lemmas. We first state a Poincar$\mathrm{\acute{e}}$-type inequality, which can be established via Fourier series expansion; a detailed derivation is given in Mei-Wong-Liu \cite{MWL-2008} and is omitted here.

\begin{lemma}\label{poincare inequality}
The set $\left\{ \sin(\frac k\pi x),\, \cos(\frac k\pi x) \right\}_{k=1}^{\infty}$ constitutes a complete orthonormal basis of the space $L^2_{\mathrm{per},0}([-1,1])$.  
Moreover, every function $f \in H^1_{\mathrm{per},0}([-1,1])$ has at least one zero in $[-1,1]$, i.e., there exists $x_0 \in [-1,1]$ such that $f(x_0) = 0$, and satisfies the Poincar$\acute{e}$ inequality:
\begin{equation}\label{property of H01per}
  \|f\| \leq \frac{1}{\pi} \|f_x\|, \quad \text{and} \quad \|f\|_{C^\gamma([-1,1])} \leq C \|f_x\|, \quad 0 < \gamma \leq \frac12.
\end{equation}
\end{lemma}

\

Next, we prove that the minimizer of the functional $G(V)$ is equivalent to a solution of the system \eqref{second-order ODE}.

\begin{lemma}\label{equivalent critical points} 
A function $V \in H^1_{\mathrm{per},0}([-1,1])$ solves system \eqref{second-order ODE} with constraint \eqref{constraint for V} if and only if it is a critical point of the functional $G(V)$ defined in \eqref{G}.
\end{lemma}
\begin{proof}
Assume $V(x)$ is a solution of system \eqref{second-order ODE} under constraint \eqref{constraint for V}. Then, for any test function $\eta(x) \in H_{\mathrm{per},0}^1([-1,1])$, multiplying equation \eqref{Approximate ns-lagrange}$_1$ by $\eta(x)$ and integrating  over $[-1, 1]$ yields
\begin{equation}\label{Variational form}
\int_{-1}^1\Big[\varepsilon^2 V_x\eta_x(x)+\big(p(\bar v) -p(V+\bar v)\big)\eta(x)\Big]dx=0,\quad\forall \eta(x) \in H_{\mathrm{per},0}^1([-1,1]),
\end{equation}
which is the variational form \eqref{Variational form}. This shows that $V(x)$ is a critical point of $G(V)$ over $H_{\mathrm{per},0}^1([-1,1])$. Conversely, suppose $V(x)$ is a critical point of $G(V)$ in $H_{\mathrm{per},0}^1([-1,1])$, i.e., it satisfies \eqref{Variational form} for all $\eta(x)\in H^1_{\mathrm{per},0}([-1,1])$ with periodic boundary condition $\eta(x-1)=\eta(x+1)$. Then, applying integration by parts and invoking the fundamental lemma of the calculus of variations, we conclude that   $V(x)$  satisfies \eqref{second-order ODE} together with the constraint \eqref{constraint for V}.
\end{proof}

\

The equivalence between the variational problem $G(V)$ and system \eqref{second-order ODE} under constraint \eqref{constraint for V} allows us to analyze multiplicity for \eqref{second-order ODE} via the minimizers of $G(V)$. We first consider the case of sufficiently large artificial viscosity, where the solution to \eqref{second-order ODE} is naturally unique. The specific conclusion is as follows:
\begin{lemma}\label{The existence of a trivial solution}
Let $\alpha$ and $\beta$ be as defined in Lemma \ref{properties of pressure}. Then, under the condition
\begin{equation}\label{uniqueness conition}
\varepsilon^2\pi^2 > \max_{v \in [\alpha, \beta]} p'(v),
\end{equation}
the steady-state system \eqref{Approximate ns-lagrange} with constraint \eqref{constraint for v,u} has $V(x) \equiv 0$ as its unique solution.
\end{lemma}
\begin{proof}By multiplying the equation \eqref{second-order ODE} by  $V(x)$ and integrating by parts over $x\in[-1,1]$, combined with Lemma \ref{poincare inequality} and \eqref{uniqueness conition}, we derive
\begin{equation}\label{inequality for solution}
\left.\begin{array}{llll}
  0&=&\displaystyle \int_{-1}^1\Big(\varepsilon^2 V_x^2+\big(\sigma+p(\bar v) -p(V+\bar v)\big)V(x)\Big)dx\\
  &=&\displaystyle\int_{-1}^1\Big(\varepsilon^2 V_x^2+\big(p(\bar v) -p(V+\bar v)\big)V(x)\Big)dx\\
  &\geq&\displaystyle\int_{-1}^1\Big(\varepsilon^2\pi^2 -p'(\bar v+\varsigma V)\Big)V^2dx\\
  &\geq&\displaystyle\int_{-1}^1\Big(\varepsilon^2\pi^2 -\max_{v\in[\alpha,\beta]}p'(v)\Big)V^2dx\geq0,
    \end{array}
  \right.
  \end{equation}
where $0<\varsigma<1$, and thus, \eqref{inequality for solution} implies that 
\begin{equation}\label{trivial solution}
  \int_{-1}^1V^2(x)dx=0,
\end{equation}
and consequently $V \equiv 0$.  This completes the proof of Lemma \ref{existence of trivial solution}.
\end{proof}

\

Based on the above Lemma \ref{The existence of a trivial solution}, we conclude that for sufficiently large artificial viscosity, the steady-state system \eqref{Approximate ns-lagrange} under constraint \eqref{constraint for v,u} admits a unique trivial solution, corresponding physically to the absence of phase transition. A natural question then arises: if the artificial viscosity is reduced, for instance, to a sufficiently small value, does the system exhibit non-trivial solutions implying the occurrence of phase transition? This question is addressed by the following lemma.

\begin{lemma}
There exists $\varepsilon_0 > 0$ such that for any $\bar{v} \in \Omega_{\mathrm{unstable}}=(\alpha, \beta)$ and every $0 < \varepsilon < \varepsilon_0$, the steady-state system \eqref{Approximate ns-lagrange} admits at least one nonconstant solution.
\end{lemma}
\begin{proof}
We begin the proof by considering the variational problem \eqref{Pe} which equivalent to system of equations \eqref{second-order ODE}.
First, we need to prove that the functional $G(V)$ \eqref{G} has a lower bound. In fact, 
\begin{equation*}
\left.\begin{array}{llll}
\displaystyle G(V)&=&\displaystyle \int_{-1}^1\Big(H(V)+\frac{\varepsilon^2}2V_{x}^2\Big)dx\\
\displaystyle &=& \displaystyle\int_{-1}^1 \Big(\big[p(\bar v)-p(\bar v+\eta_1 V)\big]V+\frac{\varepsilon^2}2V_x^2\Big)dx\\
\displaystyle &=& \displaystyle \int_{-1}^1 \Big(-\eta_1 p'(\bar v+\eta_1\eta_2V)V^2+\frac{\varepsilon^2}2V_x^2\Big)dx\\
\displaystyle &\geq& \displaystyle\delta\int_{-1}^1 \frac{\varepsilon^2}2V_x^2dx-\frac1\delta\Big(\max_{v\in[\alpha,\beta]}p'(v)\Big)^2\\
\displaystyle &\geq&\displaystyle  -\frac1\delta\Big(\max_{v\in[\alpha,\beta]}p'(v)\Big)^2,
\end{array}\right.
\end{equation*}
where $0<\eta_1,\eta_2<1,\delta>0$ are  positive constants, and the lower boundedness of the functional $G(V)$  \eqref{G} has thus been proved. 

Next, since the functional $G$ has a lower bound, it must have an infimum. According to the definition of infimum,  there must exist a sequence a minimizing sequence $V_n(x)$  for the functional $G$ defined in \eqref{G},   such that
\begin{equation}\label{minimizing sequence}
  G(V_n)\rightarrow g = \inf_{V\in H_{\mathrm{per},0}^1(\mathbb{R})}G(V).
\end{equation}
Since  $V_n$ is bounded in $H^1_{\mathrm{per},0}$, by using Sobolev embedding theorem,  there exists a subsequence (still denoted by $V_n$), such that
\begin{equation}\label{Convergence of the minimal sequence}\left.\begin{array}{llll}
\displaystyle  V_n\rightharpoonup\bar V, \qquad\mathrm{weakly\ in}\quad H^1_{\mathrm{per},0}([-1,1]),\\
\displaystyle  V_n\rightarrow\bar V, \qquad\mathrm{strongly\ in}\quad C_{\mathrm{per}}^\gamma([-1,1]), \ 0<\gamma\leq1,
\end{array}\right.
\end{equation}
which implies that
\begin{equation}\label{Integral Inequality-1}
\left.\begin{array}{llll}
\displaystyle  \int_{-1}^1 \bar V_x^2dx\leq\liminf_{n\rightarrow +\infty}\int_{-1}^1 |(V_n)_x|^2dx,\\
\displaystyle \int_{-1}^1\frac1{V_n(x)+\bar v}dx\rightarrow\int_{-1}^1\frac1{\bar V(x)+\bar v}dx,
\end{array}\right.\end{equation}
and
\begin{equation}\label{Integral Inequality-2}
\left.\begin{array}{llll}
\displaystyle 0&=&\displaystyle-\lim_{n\rightarrow +\infty}\ln\frac{1}{2}\int_{-1}^1\big(\frac{V_n(x)+\bar v-b}{\bar V(x)-\bar v-b}\big)dx\\
\displaystyle&\leq&\displaystyle-\lim_{n\rightarrow +\infty}\int_{-1}^1\frac12\ln\frac{V_n(x)+\bar v-b}{\bar V(x)+\bar v-b}dx\\
\displaystyle&=&\displaystyle-\lim_{n\rightarrow +\infty}\int_{-1}^1\frac12\ln\big(V_n(x)+\bar v-b\big)dx+\int_{-1}^1\frac12\ln\big(\bar V(x)+\bar v-b\big)dx,
\end{array}\right.\end{equation}
then applying the definition of  $G$ \eqref{G}, one has
\begin{equation*}
  G(\bar V)\leq G(V_n).
\end{equation*}
From \eqref{minimizing sequence} and \eqref{Integral Inequality-1}--\eqref{Integral Inequality-2}, it follows that $\bar V(x)$ is a minimum element of the functional $G$,
and by Lemma \ref{equivalent critical points},  it is also a solution to the stationary problem \eqref{Approximate ns-lagrange}. 
 
Finally, it remains to prove that $G(V) < 0$, implying the solution $\bar V (x)\not\equiv 0$ is nontrivial. Let $ \lambda > 0$ denote the first nonzero eigenvalue of the following periodic boundary eigenvalue problem
\begin{equation}\label{the first eigenvalue}
\left\{\begin{array}{llll}
\displaystyle
  -v_{xx}=\lambda v,& x\in\mathbb{R},\\
 \displaystyle v(x-1)=v(x+1),& x\in\mathbb{R},
 \end{array}
 \right.  
\end{equation}
and $\sin\frac  x\pi$ be one of the corresponding eigenfunctions such that $\|\sin\frac x\pi\|=1$. Since $\bar v \in (\alpha,\beta)$, we can choose $\varepsilon_1 > 0$ sufficiently small so that $[\bar v-\varepsilon_1,\bar v+\varepsilon_1] \subset (\alpha,\beta)$. Then we have the following inequality:
\begin{equation}\label{Minimum value of G}
\left.\begin{array}{llll}
\displaystyle G\big(\varepsilon_1\sin\frac x\pi\big)&=&\displaystyle \int_{-1}^1\Big(H(\varepsilon_1\sin\frac x\pi)+\frac{\varepsilon_1^2\varepsilon^2}{2\pi^2}\cos^2\frac x\pi\Big)dx\\
\displaystyle &=& \displaystyle\varepsilon_1\int_{-1}^1 \Big(\big[p(\bar v)-p(\bar v+\varepsilon_1\eta_1 \sin\frac  x\pi)\big]\sin\frac  x\pi\Big) dx+\frac{\varepsilon^2_1\varepsilon^2}{2\pi^2}\\
\displaystyle &\leq& \displaystyle \varepsilon^2_1\Big(-\eta_1 \min_{\zeta\in[\bar v-\varepsilon_1,\bar v+\varepsilon_1]}p'(\zeta)+\frac{\varepsilon^2}{2\pi^2}\Big),
\end{array}\right.
\end{equation}
where $0<\eta_1,\eta_2<1$ are the positive constants. 
Having fixed this $\varepsilon_1$ in \eqref{Minimum value of G}, we now choose $\varepsilon$ sufficiently small so that $G(\varepsilon_1 \sin \frac{x}{\pi}) < 0$, which in turn implies $G(\bar V) < 0$.

\end{proof}

\

This completes the proof of Theorem \ref{existence of trivial solution}. These results show that the multiplicity of solutions is jointly determined by the artificial viscosity coefficient and the relative position of $\bar v$ with respect to the Maxwell region $\Omega_{\mathrm{Maxwell}} = (\alpha_0, \beta_0)$.

\

\subsection{Proof of Theorem 1.3. }\label{Proof-2-3}
\ \ \ \ As shown in the preceding analysis in Section \ref{Proof-2-2}, the system \eqref{Approximate ns-lagrange} admits multiple steady-state solutions when $\bar v$ lies in the unstable region $[\alpha, \beta]$ and the artificial viscosity $\varepsilon$ is sufficiently small. By the equivalence between the constrained system \eqref{second-order ODE} and the variational problem \eqref{G}, the physically relevant solution corresponds to the minimizer of the energy functional. Guided by this principle, we introduced in Section \ref{sec:int} the notion of two-smooth-interface phase transition solution. In what follows, we carry out a detailed study of such  solutions and present a complete proof of Theorem \ref{thm:approximation}. In fact, Theorem \ref{thm:approximation} will be established through Proposition \ref{prop:two-interface-phase-transition}--\ref{unique}.

We first restrict attention to the energy-minimizing solutions among those obtained in section \ref{Proof-2-2} and analyze their properties. For this purpose, we introduce a scaling transformation for system \eqref{Approximate ns-lagrange} defined as follows:
\begin{equation}\label{scaling}
  y=\frac{x}{\varepsilon}.
\end{equation}
Thus, under the scaling transformation \eqref{scaling}, the system \eqref{Approximate ns-lagrange} reduces, via the antiderivative method, to
\begin{equation}\label{ODE under scaling transformation}
\left\{\begin{array}{llll}
\displaystyle -\frac{d^2v}{dy^2}=p(v)-\sigma,&&y\in\mathbb{R}, \\
\displaystyle  v(y-\frac{1}{\varepsilon})=v(y+\frac{1}{\varepsilon}),&& y\in\mathbb{R},\\
\displaystyle \int_{-\frac{1}{\varepsilon}}^{\frac{1}{\varepsilon}}vdy=\frac{2\bar v}{\varepsilon}.
\end{array}\right.
\end{equation} 
where $\sigma$ (= constant) is a Lagrange multiplier.

We proceed to a phase plane $(v,v')$ analysis of system \eqref{ODE under scaling transformation}. The first integral may be derived by multiplying system \eqref{ODE under scaling transformation}$_1$ by $v_y$, giving:
\begin{equation}\label{first integral}
\frac{1}{2}\Big(\frac{dv}{dy}\Big)^2=W (v)+\sigma v-\lambda \xlongequal{\mathrm{def}}f_{\sigma,\lambda}(v),
\end{equation}
where $\lambda$ is a constant and $W(v)$ defined as \eqref{W}. The periodic boundary value condition \eqref{ODE under scaling transformation}$_2$ then implies
\begin{equation}\label{Two-point boundary value condition}
  \left\{\begin{array}{llll}
  \displaystyle\big(W (v(y))+\sigma v(y)\big)\Big|_{y=-\frac{1}{\varepsilon}}=\big(W (v(y))+\sigma v(y)\big)\Big|_{y=\frac{1}{\varepsilon}}=\lambda,\\
  \displaystyle v\Big|_{y=-\frac{1}{\varepsilon}}=v\Big|_{y=\frac{1}{\varepsilon}}.
  \end{array}\right.
\end{equation}
\begin{figure}
\centering
\subfigure[\scriptsize{$\sigma>\sigma_0$}]
{
\begin{minipage}[t]{0.3\textwidth}
\centering
\includegraphics[width=1.4 in,height=1.7 in]{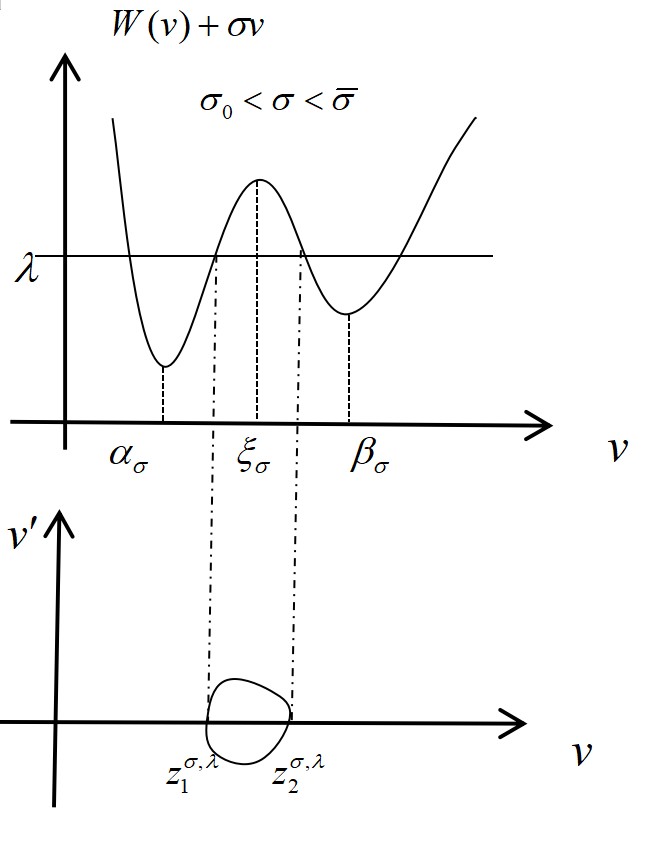}
\end{minipage}}
\hfill\centering
\subfigure[\scriptsize{$\sigma=\sigma_0$ }]{
\begin{minipage}[t]{0.3\textwidth}
\centering
\includegraphics[width=1.2 in,height=1.7 in]{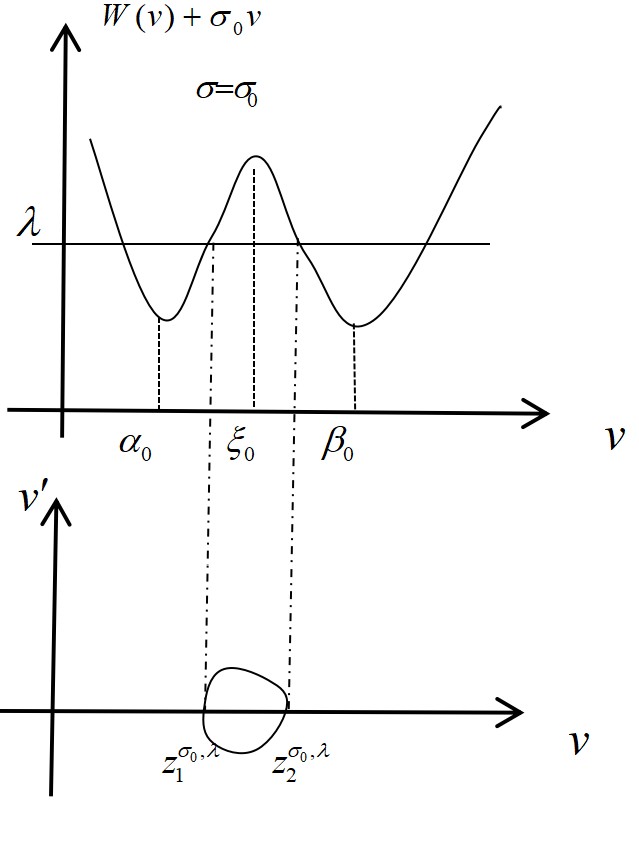}
\end{minipage}}
\subfigure[\scriptsize{$\sigma<\sigma_0$ }]{
\begin{minipage}[t]{0.3\textwidth}
\centering
\includegraphics[width=1.2 in,height=1.7 in]{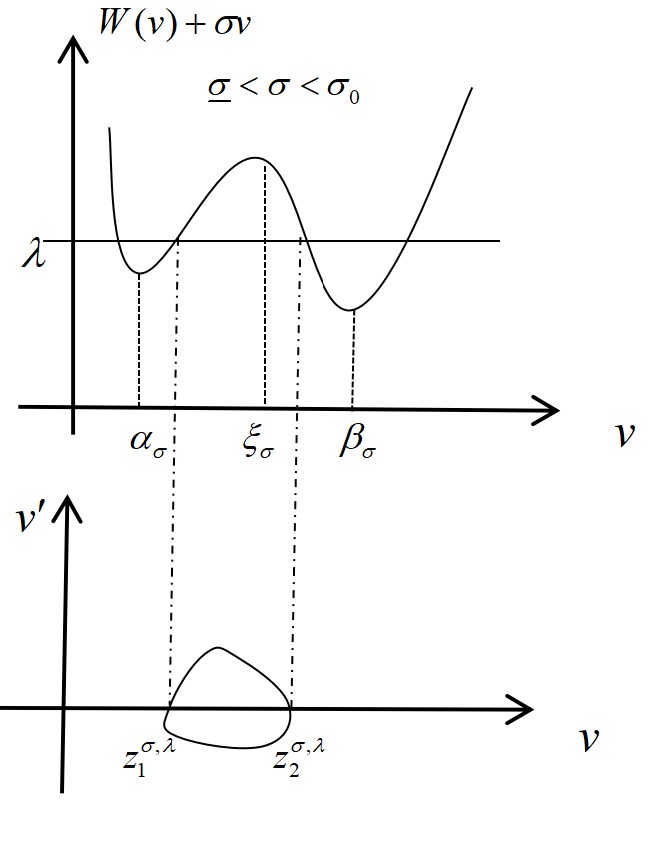}
\end{minipage}}\caption{image of $W(v)+\sigma v$ and phase plane of $v$--$v'$.}
\label{model2figure1}
\end{figure}

The relationship between $v$ and $v'$  is given by equation \eqref{first integral}.   To establish the existence of periodic solutions for the steady-state system \eqref{Approximate ns-lagrange} via the phase plane method, we first examine several properties of the functional $W(v)+\sigma v$, which can be derived using basic calculus. These properties are summarized briefly as follows.

\begin{lemma}\label{properties of w} Assume the predetermined fixed temperature is in the subcritical state \eqref{Critical temperature condition}. Then for each $\sigma \in (\underline{\sigma}, \overline{\sigma})$, the function $W(v)+\sigma v$ satisfies the following:
  \begin{enumerate}

\item[(i)] It has three positive critical points $\beta_\sigma > \xi_\sigma > \alpha_\sigma > b$, where $\beta_\sigma$ and $\alpha_\sigma$ are local minima, $\xi_\sigma$ is a local maximum, and
\begin{equation}\label{extremum points of W-lv}
p(\beta_\sigma) = p(\xi_\sigma) = p(\alpha_\sigma) = \sigma.
\end{equation}
    
 \item[(ii)] Its derivative satisfies the monotonicity conditions: 
 \begin{equation}\label{monotinicity of W}
 \left\{\begin{array}{llll}
 \displaystyle  \big(W(v)+\sigma v\big)'<0, \qquad  \mathrm{on}\ (b,\alpha_\sigma)\cup(\xi_\sigma,\beta_\sigma),\\
  \displaystyle  \big(W(v)+\sigma v\big)'>0,\qquad \mathrm{on}\ (\alpha_\sigma,\xi_\sigma)\cup(\beta_\sigma,+\infty).
 \end{array}\right.
 \end{equation}

 \item[(iii)] For $\sigma \in [\underline{\sigma}, \overline{\sigma}]$, the values at the local minima are ordered as follows: 
 \begin{equation}\label{properity-3 of W}
 \left\{\begin{array}{llll}
   \displaystyle W(\beta_\sigma)+\sigma \beta_\sigma>W(\alpha_\sigma)+\sigma \alpha_\sigma, &\sigma_0<\sigma<\overline{\sigma}, \\ 
   \displaystyle W(\beta_{\sigma_0})+\sigma_0 \beta_{\sigma_0}=W(\alpha_{\sigma_0})+\sigma_0 \alpha_{\sigma_0}, & \sigma=\sigma_0,\\
   \displaystyle W(\beta_\sigma)+\sigma \beta_\sigma<W(\alpha_\sigma)+\sigma \alpha_\sigma,& \underline{\sigma}<\sigma<\sigma_0.
\end{array}\right.
\end{equation}

 \end{enumerate}
\end{lemma}
\begin{proof}
The critical points $\beta_\sigma$, $\xi_\sigma$, and $\alpha_\sigma$ of $W(v)+\sigma v$ are determined by solving $W'(v) =- \sigma$. Substituting the expression of $W(v)$ from \eqref{W} into this equation yields the characterization of the extremum points given in \eqref{extremum points of W-lv}. Statements (i) and (ii) then follow directly from the properties of $p(v)$.
To show (iii), consider the identity
\begin{equation}\label{monotonicity of W}
\frac{d}{d\sigma}\left[\big( W(\alpha_\sigma) + \sigma \alpha_\sigma\big) - \left(W(\beta_\sigma) + \sigma \beta_\sigma\right)\right] =-\big( \beta_\sigma - \alpha_\sigma\big)<0.
\end{equation}
Together with \eqref{Maxwell condition}, this establishes the desired inequalities, thereby completing the proof of lemma \ref{properties of w}.
\end{proof}

\


Because of the periodic boundary condition \eqref{ODE under scaling transformation}$_2$, our problem reduces to finding a 
limit cycle  which has 
duration $\frac{2}{\varepsilon}$. Thus nonconstant periodic solutions of \eqref{ODE under scaling transformation}$_2$ are  possible only for $\sigma\in[\underline{\sigma},\overline{\sigma}]$. We therefore restrict our attention to this interval. Further, it is clear  that for such $\sigma$, the constant $\lambda$ will have the range
\begin{equation}\label{range of sigma and lambda}
\left\{\begin{array}{llll}
 \displaystyle W(\beta_\sigma)+\sigma \beta_\sigma<\lambda<W(\xi_{\sigma})+\sigma \xi_{\sigma}, &\displaystyle \sigma_0<\sigma< \overline{\sigma},\\
\displaystyle W(\alpha_{\sigma})+\sigma \alpha_{\sigma}<\lambda<W(\xi_\sigma)+\sigma \xi_\sigma, &\displaystyle\underline{\sigma}<\sigma\leq\sigma_0.
\end{array}\right.
\end{equation}
Therefore, given fixed parameters $\varepsilon > 0$ and $\bar{v} \in (b, +\infty)$, we now specify the domain of $\lambda$ and $\sigma$:
\begin{equation}\label{the set of admissible pair}
D^{\varepsilon} = \Big\{ (\sigma, \lambda) \Big| \underline{\sigma} < \sigma < \overline{\sigma},\ \lambda\ \text{satisfies\ } \eqref{range of sigma and lambda} \Big\}.
\end{equation}
The boundary of $D^{\varepsilon}$, denoted $\partial D^\varepsilon$, decomposes into three mutually disjoint subsets:
\begin{equation}\label{boundary of the set}
  \partial D^{\varepsilon}=\bigcup_{i=1}^3\Gamma_i^{\varepsilon},
\end{equation}
where
\begin{equation}\label{Gamma-i}\left.\begin{array}{llll}
 \displaystyle \Gamma^{\varepsilon}_1=\big\{(\sigma,\lambda)\big|\lambda=W(\beta_\sigma)+\sigma \beta_\sigma,\sigma_0\leq\sigma<\bar\sigma\big\},\\
 \displaystyle \Gamma^{\varepsilon}_2=\big\{ (\sigma,\lambda)\big|\lambda=W(\alpha_{\sigma})+\sigma \alpha_{\sigma},\underline\sigma\leq\sigma<\sigma_0\big\},\\
 \displaystyle \Gamma^{\varepsilon}_3=\big\{(\sigma,\lambda)\big|\lambda=W(\xi_{\sigma})+\sigma \xi_{\sigma},\underline\sigma\leq\sigma\leq\bar\sigma\big\}.
\end{array}\right.\end{equation}

By Lemma \ref{properties of w}, for each  $(\sigma,\lambda)\in D^{\varepsilon}$, there exist exactly two points, $z_1^{\sigma,\lambda} \in (\alpha_{\sigma}, \xi_\sigma)$ and $z_2^{\sigma,\lambda} \in (\xi_\sigma, \beta_\sigma)$, which are the two roots of the following equation:
\begin{equation}
f_{\sigma,\lambda}(v) = W(v) + \sigma v - \lambda = 0.
\end{equation}

 For $\alpha_0 < \bar v < \beta_0$, all periodic solutions of \eqref{first integral} correspond to closed orbits. Furthermore, when the boundary conditions in \eqref{Two-point boundary value condition} are satisfied, each solution begins and ends at one of the points $(z_1^{\sigma,\lambda}, 0)$ or $(z_2^{\sigma,\lambda}, 0)$. As shown in Figure 3.
To formalize the analysis, we define a \textbf{transition point} as any point $y \neq \frac{1}{\varepsilon}$ at which $v'(y) = 0$.
According to Definition \ref{two-smooth-interface-phase-transition-solution}, a solution is referred to as a \textbf{two-smooth-interface phase transition solution} if it possesses exactly two transition points. Hence, applying the substitution method for definite integrals to such a phase transition solution with two interfaces yields the following key identity for constructing periodic solutions of system \eqref{Approximate ns-lagrange}:
\begin{equation}
\int_{-\frac 1\varepsilon}^{\frac 1\varepsilon}v^n(y)dy
=\frac{2}{\sqrt{2}}\int_{z_1^{\sigma,\lambda}}^{z_2^{\sigma,\lambda}}s^n\big( f_{\sigma,\lambda}(s)\big)^{-\frac12}ds.
\end{equation}
Furthermore, using the conditions that the period of this periodic solution is $\frac2\varepsilon$ and the constraint  \eqref{constraint for v,u} $v$,
 the phase transition solution with two interfaces $v$ imposes the following additional restrictions on  $(\sigma, \lambda)$:
\begin{equation}\label{restriction for admissible pair}
\left\{\begin{array}{llll}
 \displaystyle I_0(\sigma,\lambda)\xlongequal{\mathrm{def}}\frac{1}{\sqrt{2}} \int_{z_1^{\sigma,\lambda}}^{z_2^{\sigma,\lambda}}\big( f_{\sigma,\lambda}(s)\big)^{-\frac12}ds=\frac{1}\varepsilon,\\ \displaystyle I_1(\sigma,\lambda)\xlongequal{\mathrm{def}}\frac{1}{\sqrt{2}}\int_{z_1^{\sigma,\lambda}}^{z_2^{\sigma,\lambda}}s\big(f_{\sigma,\lambda}(s)\big)^{-\frac12}ds=\frac{\bar v}\varepsilon.
\end{array}\right.
\end{equation}
Hence, a necessary and sufficient condition for  $(\sigma, \lambda)\in D^{\varepsilon} (\sigma, \lambda)$ to correspond to a phase transition solution with two interfaces $v(y)$ is that it satisfies the system \eqref{restriction for admissible pair}, yielding a one-to-one correspondence.

\begin{lemma} Given fixed parameters $\varepsilon > 0$ and $\bar{v} \in (\alpha_0, \beta_0)$, a one-to-one correspondence exists between phase transition solution with two interfaces $v(y)$ and  $(\sigma,\lambda)\in D^{\varepsilon} $ satisfying the restriction in \eqref{restriction for admissible pair}. Specifically, each phase transition solution with two interfaces uniquely determines an associated  $(\sigma,\lambda)\in D^{\varepsilon} $ that satisfies \eqref{restriction for admissible pair}; and conversely, every  $(\sigma,\lambda)\in D^{\varepsilon} $ fulfilling this restriction \eqref{restriction for admissible pair} corresponds to exactly one phase transition solution with two interfaces  $v(y)$.
\end{lemma}

\begin{remark}
Extending \eqref{restriction for admissible pair}, we find that for any solution $v(y)$ of \eqref{first integral} and \eqref{Two-point boundary value condition} with $2N$ transitions ($N\geq1$), there exists a corresponding  $(\sigma, \lambda)\in D^{\varepsilon} (\sigma, \lambda)$ such that
\begin{equation}\label{The necessary condition for the number of phase transitions}
\left\{\begin{array}{llll}
\displaystyle \frac{N}{\sqrt{2}}\int_{z_1^{\sigma,\lambda}}^{z_2^{\sigma,\lambda}}\Big( -\big(\frac{a}{s}-\frac{a}{\bar{v}}\big)-R\theta\ln\big(\frac{s-b}{\bar v-b}\big)+\sigma s-\lambda\Big)^{-\frac12}ds=\frac{1}\varepsilon,\\
\displaystyle \frac{N}{\sqrt{2}}\int_{z_1^{\sigma,\lambda}}^{z_2^{\sigma,\lambda}}s\Big( -\big(\frac{a}{s}-\frac{a}{\bar{v}}\big)-R\theta\ln\big(\frac{s-b}{\bar v-b}\big)+\sigma s-\lambda\Big)^{-\frac12}ds=\frac{\bar v}\varepsilon.
\end{array}\right.
\end{equation}
By applying the two aforementioned constraints \eqref{The necessary condition for the number of phase transitions}, we can derive a precise estimate of the number of periodic solutions that are mathematically feasible. However, not all of these potential solutions carry physical relevance. To identify those that are physically admissible, we will adopt an energy minimization principle. Specifically, in the final part of this section, we will demonstrate that among all possible nontrivial periodic solutions, only those involving exactly two phase transitions correspond to configurations with minimal energy and are thus physically realizable.
\end{remark}

\

In the following, we first establish the existence of a two-smooth-interface phase transition solution under constraint condition \eqref{restriction for admissible pair}, and then prove that such solutions achieve minimal energy among all possible periodic solutions. Further properties of $W(v)+\sigma v$ and the properties of singular integrals of $I_0,I_1$ in \eqref{restriction for admissible pair}, are stated in lemmas \ref{lem:7}-\ref{nr-property}. The proofs of these lemmas can be obtained through elementary calculus. To avoid making the text overly lengthy, we only provide a general outline of the proofs here. For detailed proofs, readers are referred to similar work in Carr, Gurtin, and Slemrod \cite{CGS1984}.

\

\begin{lemma}\label{lem:7}  The critical points $\alpha_\sigma$, $\xi_\sigma$, $\beta_\sigma$ of $W(v)+\sigma v$, regarded as functions of $\sigma$, are continuous on $[\underline{\sigma},\overline{\sigma}]$ and $C^4$ on $(\underline{\sigma},\overline{\sigma})$, with $\alpha_\sigma$ and $\beta_\sigma$ strictly decreasing and $\xi_\sigma$ strictly increasing. Moreover, 
\begin{equation}
\alpha_{\underline{\sigma}}=\xi_{\underline{\sigma}}=\alpha,\qquad \beta_{\overline{\sigma}}=\xi_{\overline{\sigma}}=\beta,
\end{equation}
and for every $\sigma\in(\underline{\sigma},\overline{\sigma})$,
\begin{equation}
  W''(\alpha_{\sigma}), \quad W''(\beta_\sigma) > 0,\quad W''(\xi_{\sigma}) < 0. 
\end{equation}
\end{lemma}

\


\begin{lemma}\label{lem:8}
Consider the function $ W(v) + \sigma v - \lambda$. Its two roots, denoted by $z_1^{\sigma,\lambda}$ and $z_2^{\sigma,\lambda}$, are continuous on the closed domain $\bar{D}^{\varepsilon}$ and $C^4$-smooth in the interior $D^{\varepsilon}$.
Furthermore, for all $(\sigma,\lambda) \in D^{\varepsilon}$, the roots are ordered as
\begin{equation}
\alpha_0 \;<\; z_1^{\sigma,\lambda} \;<\; z_2^{\sigma,\lambda} \;<\; \beta_0.
\end{equation}
On the boundary curves $\Gamma^{\varepsilon}_i$, the roots satisfy the identifications
\begin{equation}
\left.\begin{aligned}
z_1^{\sigma,\lambda} &=\alpha_\sigma , & & (\sigma,\lambda)\in\Gamma^{\varepsilon}_1,\\
z_2^{\sigma,\lambda} &=\beta_{\sigma},  & & (\sigma,\lambda)\in\Gamma^{\varepsilon}_2,\\
z_1^{\sigma,\lambda} &= z_2^{\sigma,\lambda}=\xi_{\sigma}, & & (\sigma,\lambda)\in\Gamma^{\varepsilon}_3.
\end{aligned}\right.
\end{equation}
In addition, the following non-vanishing conditions hold outside the point $(\sigma_0,\lambda_0)$:
\begin{equation}
\left.\begin{aligned}
p(z_1^{\sigma,\lambda}) - \sigma &\neq 0, && (\sigma,\lambda)\in\Gamma^{\varepsilon}_2\setminus\{(\sigma_0,\lambda_0)\},\\
p(z_2^{\sigma,\lambda}) - \sigma &\neq 0, && (\sigma,\lambda)\in\Gamma^{\varepsilon}_1\setminus\{(\sigma_0,\lambda_0)\}.
\end{aligned}\right.
\end{equation}
\end{lemma}

\ 

Now we denote the transformation
\begin{equation}\label{h-Phi}\left\{\begin{array}{llll}
 \displaystyle h_1\xlongequal{\mathrm{def}}\lambda-W(\alpha_\sigma)-\sigma\alpha_\sigma,\qquad
  h_2\xlongequal{\mathrm{def}}\lambda-W (\beta_\sigma)-\sigma\beta_\sigma,\\
  \displaystyle (h_1,h_2) \xlongequal{\mathrm{def}}\mathbf{\Psi}(\sigma,\lambda)=(\Psi_1,\Psi_2),
\end{array}\right.\end{equation}

\ 

In order to establish the direct  relationship between the roots of  $W(v) + \sigma v - \lambda$ and the critical points of $W(v) + \sigma v$, we need to introduce the following so-called \textbf{nearly regular} definition which introduced by Carr--Gurtin--Slemrod \cite{CGS1984}.
\begin{definition} (\textbf{Nearly regular})
Let $f(h_1,h_2)$ be a function defined on $\mathbf{\Psi}(D^{\varepsilon})$ near $(h_1,h_2)=0$. We say $f$ is nearly regular at $(h_1,h_2)=0$ if it is of class $C^1$ in that region and satisfies the following asymptotic behaviour: there exists a constant $s \in (0,1)$ such that
\begin{equation}
f(h_1,h_2) = f(0,0) + O\big((h_1^2+h_2^2)^{1-s}\big) \quad \text{and} \quad
\nabla f(h_1,h_2) = \big( O(h_1^{-s}),\, O(h_2^{-s}) \big),
\end{equation}
as $(h_1,h_2) \rightarrow (0,0)$ within $(h_1,h_2)\in\mathbf{\Psi}(D^{\varepsilon})$. Furthermore, we denote
\begin{equation}
  NR(\Omega)=\big\{g\in C(\Omega)\big|g(\mathbf{\Psi}^{-1})\ \mathrm{is\ nearly\ regular\ at\ }(0,0)  
  \big\},
\end{equation}and
\begin{equation}
  NR^+(\Omega)=\big\{g\in NR(\Omega)\big|g>0 \ \mathrm{on}\ \Omega\big\}. 
\end{equation}
\end{definition}

\begin{lemma}\label{lem:9}
The functions $z_1^{\sigma,\lambda}$ belong to $NR(\overline{D}^{\varepsilon})$. In fact, for $(h_1,h_2)=\mathbf{\Psi}(\sigma,\lambda)$,
\begin{equation}
  z_1^{\sigma,\lambda}=\alpha_\sigma+h_1^{\frac12}\Theta_1(\sigma,\lambda),\qquad z_2^{\sigma,\lambda}=\beta_\sigma+h_2^{\frac12}\Theta_2(\sigma,\lambda),
\end{equation}
on $D^{\varepsilon}$ with $\Theta_i\in NR^+(\bar{D}^{\varepsilon}\setminus\{(\underline{\sigma},\underline{\lambda}),(\overline{\sigma},\bar{\lambda})\})$, where $\underline{\lambda}=W(\beta)+\underline\sigma\beta$ and $\overline{\lambda}=W(\alpha)+\overline\sigma\alpha$.
\end{lemma}

\

Lemma \ref{the properties of I} is crucial for providing the proof of the existence of implicit functions $(\sigma,\lambda)$ in the system of \eqref{restriction for admissible pair}. To achieve this, the following definition is first presented: 
\begin{definition}
Define
  \begin{equation}\label{W-i}
   \mathcal{W}_i\xlongequal{\mathrm{def}}\big\{(\sigma,\lambda)\in\overline D^{\varepsilon}\big| W''(z_1^{\sigma,\lambda})>0 \big\}.
\end{equation}
and let $\Omega_0,\Omega_1,\Omega_2$ be arbitrary closed subregions of $\overline D^{\varepsilon}$  satisfying
\begin{equation}\label{Omega-i}
  \Omega_0\subset \mathcal{W}_1\cap \mathcal{W}_2,\ \ \Omega_1\subset(\mathcal{W}_1\setminus\Gamma^{\varepsilon}_2),\ \ \Omega_2\subset(\mathcal{W}_1\setminus\Gamma^{\varepsilon}_1).
\end{equation}
\end{definition}

\

\begin{lemma}\label{the properties of I}
For $n=0,1$,  the following hold:
 \begin{equation}
 \left.\begin{array}{llll}\label{The regularity of the function I}
 \displaystyle   I_n+\big[2W''(z_i^{\sigma,\lambda})\big]^{-\frac12} (z_i^{\sigma,\lambda})^n\ln \Psi_i\in C(\Omega_i),\quad \forall i=1,2,\\ 
 \displaystyle I_n+\sum_{i=1}^2\big[2W''(z_i^{\sigma,\lambda})\big]^{-\frac12} (z_i^{\sigma,\lambda})^n\ln \Psi_i\in NR(\Omega_0)\ \ \forall i=1,2.
 \end{array}\right.\end{equation}
\end{lemma}
\begin{proof}
The integrals $I_0$ and $I_1$ defined in \eqref{restriction for admissible pair} have integrands that are singular at $v = v_i^{\sigma,\lambda}$. We therefore fix a small $a > 0$ and decompose each $I_n$ as
\begin{equation}\label{The decomposition form of the function I}
I_n = J_{n1} + J_{n2} + J_{n3} \qquad (n=0,1),
\end{equation}
where
\begin{equation*}
  J_{n1} =\int_{z_1^{\sigma,\lambda}}^{z_1^{\sigma,\lambda}+\eta}s^n\big( f_{\sigma,\lambda}(s)\big)^{-\frac12}ds,  
  J_{n2} = \int_{z_2^{\sigma,\lambda}-\eta}^{z_2^{\sigma,\lambda}}s^n\big( f_{\sigma,\lambda}(s)\big)^{-\frac12}ds, 
  J_{n3}= \int_{z_1^{\sigma,\lambda}+\eta}^{z_2^{\sigma,\lambda}-\eta}s^n\big( f_{\sigma,\lambda}(s)\big)^{-\frac12}ds. 
\end{equation*}
Applying Taylor's formula with remainder yields the estimates stated in \eqref{The regularity of the function I}, moreover, 
\begin{equation}\label{J01}
  J_{01}=\frac{2}{\sqrt{2W''(z_1^{\sigma,\lambda})}}\Big(\ln \big(\eta+\varsigma+\big(\eta^2+2\varsigma\eta\big)^{\frac12}\big)-\ln\varsigma\Big)+g(\sigma,\lambda,\eta),
\end{equation}
where 
$$\varsigma=\frac{W'(v)-W'(\alpha_\sigma)}{W''(v)},$$
 and $g(\sigma,\lambda,\eta)$ continuous on $\Omega_1\cup\Omega_0\times[0,\eta]$.
 The detailed computations are similar to those given in Carr--Gurtin--Slemrod \cite{CGS1984} and are therefore omitted here.
\end{proof}

\

To establish the exponential decay estimate  of $\bigl| z_1^{\sigma,\lambda} - \alpha_0 \bigr|$ and $\bigl| z_2^{\sigma,\lambda} - \beta_0 \bigr|$,  we employ a scaling transformation characterized by
\begin{equation}\label{transformation of h}
h_i = e^{\mu_i k_i}e^{ - \frac{c_i(\bar v)}{\varepsilon}}, \qquad i=1,2,
\end{equation}
valid for $\mathbf{k} = (k_1, k_2) \in \mathbb{R}^2$, $\varepsilon > 0$, and $\bar{v} \in (\alpha_0, \beta_0)$. The parameters are specified by
\begin{equation}\label{parameters of h}
\left.\begin{array}{llll}
 \displaystyle \mu_i=\big[B_i(\beta_0-\alpha_0)\big]^{-1}, B_1=[2W''(\alpha_0)]^{-\frac12},B_2=[2W''(\beta_0)]^{-\frac12},\\
 \displaystyle c_1(\bar v)=\frac{2(\beta_0-\bar v)}{B_1(\beta_0-\alpha_0)}>0,\qquad c_2(\bar v)=\frac{2(\bar v-\alpha_0)}{B_2(\beta_0-\alpha_0)}>0,
\end{array}
\right.\end{equation}
with $\bar{v}$ situated in the Maxwell construction region. We now compute the corresponding pair $(\sigma, \lambda)$ from \eqref{h-Phi}$_2$,  that is,
\begin{equation}\label{Inverse function-1}
  (\sigma, \lambda) = \mathbf{\Psi}^{-1}(h_1,h_2),\qquad \forall (h_1, h_2) \in \mathcal{H},
\end{equation}
substituting\eqref{transformation of h} into\eqref{Inverse function-1}, one has
\begin{equation}\label{Inverse function}
(\sigma, \lambda) = \Psi^{-1}\big(e^{\mu_1k_1}e^{-\frac{c_1(\bar v)}{\varepsilon}}, e^{\mu_2k_2}e^{-\frac{c_2(\bar v)}{\varepsilon}}\big).
\end{equation}
Conversely, for any $(\sigma, \lambda) \in D^{\varepsilon,\bar{v}}$ and $\varepsilon > 0$, the vector $\mathbf{k}=(k_1,k_2)$ is determined by solving equation \eqref{Inverse function}. Applying the implicit function theorem in this context, however, poses a significant challenge, because the integrand in the definition of $I_n$ (see \eqref{restriction for admissible pair}) becomes singular at the points $z_1^{\sigma,\lambda}$ and $z_2^{\sigma,\lambda}$. Before establishing the existence of the inverse function $\mathbf{k}=(k_1,k_2)$, we first outline several properties of the ``\textbf{nearly regular}'' class of functions introduced by Carr-Gurtin-Slemrod \cite{CGS1984}.

\

\begin{lemma}\label{nr-property}
Let
\begin{equation}
f = f_0 + f_1\ln\Psi_1 + f_2\ln\Psi_2,
\end{equation}
where
\begin{equation}
f_0,f_1,f_2 \in NR(\Omega_0), \qquad f_1(\sigma_0,\lambda_0)=f_2(\sigma_0,\lambda_0)=0.
\end{equation}
Then there exists a neighborhood $\mathcal{H}$ of $\mathcal{E}_0$ such that $f(\sigma(k_1,k_2,\varepsilon,\bar v),\lambda(k_1,k_2,\varepsilon,\bar v))$ admits a $C^1$ extension to $\mathcal{H}$ and $\nabla_{k_1,k_2,\varepsilon,\bar v} f= 0$ on $\mathcal{E}_0$. Here
\begin{equation}\label{H-E0}
\begin{aligned}
\mathcal{E}_0 &= \big\{(k_1,k_2,\varepsilon,\bar v) \mid (k_1,k_2)\in\mathbb{R}^2,\varepsilon=0,\bar v\in(\alpha_0,\beta_0) \big\},\\
\mathcal{H} &\text{ is a neighborhood of } \mathcal{E}_0,\\
\mathcal{H}^+ &= \big\{(k_1,k_2,\varepsilon,\bar v) \mid \varepsilon > 0 \big\}.
\end{aligned}
\end{equation}
\end{lemma}

\

By applying Lemmas \ref{the properties of I} and \ref{nr-property}, we obtain the following key proposition \ref{existence of k1 and k2}, which ensures the existence of the implicit functions $\mathbf{k} = (k_1, k_2)$. Consequently, the \textbf{two-smooth-interface phase transition solution} \( v^{\varepsilon}(y) \)  of system \eqref{Approximate ns-lagrange} exists and is unique.

\


\begin{proposition}\label{existence of k1 and k2}
There exist a neighborhood $\mathcal{H}$ of $\mathcal{E}_0$ and a $C^1$ function $\mathbf{F}=(F_1,F_2)$ on $\mathcal{H}$ such that $\nabla F_i = 0$ on $\mathcal{E}_0$ for $i=1,2$. For $(k_1,k_2,\varepsilon,\bar v)\in \mathcal{H}^+$ and $(\sigma,\lambda)=\Phi^{-1}(h_1,h_2)$, the equation
\begin{equation}\label{existence of k}
\mathbf{k}=(k_1,k_2) = \mathbf{F}(k_1,k_2,\varepsilon,\bar v)
\end{equation}
holds if and only if $(\sigma,\lambda)$ solves \eqref{restriction for admissible pair}.

Moreover, given $\delta>0$, there exist a neighborhood $B_\delta$ of $\mathbf{a}\in\mathbb{R}^2$ and a constant $\bar\varepsilon=\bar\varepsilon_\delta>0$ such that $B_\delta\times[-\bar\varepsilon,\bar\varepsilon]\times [\alpha_0+\delta,\beta_0-\delta]\subset \mathcal{H}$, and a $C^1$ function
\begin{equation}\label{k-delta}
\mathbf{k}_\delta: [-\bar\varepsilon,\bar\varepsilon]\times [\alpha_0+\delta,\beta_0-\delta] \to B_\delta,
\end{equation}
satisfying: for each $(\varepsilon,\bar v)\in[-\bar\varepsilon,\bar\varepsilon]\times [\alpha_0+\delta,\beta_0-\delta]$, equation \eqref{existence of k} has a unique solution $\mathbf{k}\in B_\delta$, given by
\begin{equation}\label{existence of k-delta}
\mathbf{k} = \mathbf{k}_\delta(\varepsilon,\bar v).
\end{equation}
\end{proposition}

\

Consequently, by combining the transformation rule \eqref{transformation of h} with the existence result for the pair \( (k_1,k_2) \) (Proposition~\ref{existence of k1 and k2}), we obtain the following exponential decay estimates for $\bigl| z_1^{\sigma,\lambda} - \alpha_0 \bigr|$ and 
$\bigl| z_2^{\sigma,\lambda} - \beta_0 \bigr|$.

\begin{definition}
Let \( v^{\varepsilon}(y) \) denote the two-smooth-interface phase transition solution of \eqref{Approximate ns-lagrange} under the constraint \eqref{constraint for v,u}. This solution is globally either single-valley or single-peak, with its monotonicity changing at exactly one point within the interval \( \big[-\tfrac{1}{\varepsilon}, \tfrac{1}{\varepsilon}\big] \). 
We focus on the globally single-valley case (the single-peak case is analogous) and define the following two “times” (i.e., increments in the independent variable \( y \)):
\begin{itemize}
    \item \( T_1^{\varepsilon} \) is the increment in \( y \) required for \( v^{\varepsilon} \) to increase from its initial value \( v^{\varepsilon}\bigl(-\tfrac{1}{\varepsilon}\bigr) \) to \( v^{\varepsilon}\bigl(-\tfrac{1}{\varepsilon}\bigr) + \varepsilon \);
    \item \( T_2^{\varepsilon} \) is the increment in \( y \) required for \( v^{\varepsilon} \) to decrease from \( v^{\varepsilon}\bigl(-\tfrac{1}{\varepsilon}\bigr) - \varepsilon \) back to \( v^{\varepsilon}\bigl(-\tfrac{1}{\varepsilon}\bigr) \).
\end{itemize}
\end{definition}

\begin{proposition}\label{prop:two-interface-phase-transition}
For any $\delta > 0$, there exist constants $\varepsilon_\delta, C_\delta > 0$, depending only on $\delta$, such that for all $\varepsilon \in (0, \varepsilon_\delta)$ and $\bar{v} \in (\alpha_0 + \delta, \beta_0 - \delta)$, the system \eqref{Approximate ns-lagrange} under the constraint \eqref{constraint for v,u} admits at least one two-smooth-interface phase transition solution of either single-valley type  or single-peak type, corresponding to parameters $(\lambda, \sigma)$. Moreover, as $\varepsilon \to 0$, the following estimates hold:
\begin{equation}\label{decay-1}
\bigl| z_1^{\sigma,\lambda} - \alpha_0 \bigr| = O\bigl(e^{-C_\delta / \varepsilon}\bigr), \qquad
\bigl| z_2^{\sigma,\lambda} - \beta_0 \bigr| = O\bigl(e^{-C_\delta / \varepsilon}\bigr).
\end{equation}
\begin{equation}\label{decay-2}
\bigl| \alpha_\sigma - \alpha_0 \bigr| = O\bigl(e^{-C_\delta / \varepsilon}\bigr), \qquad
\bigl| \beta_\sigma - \beta_0 \bigr| = O\bigl(e^{-C_\delta / \varepsilon}\bigr).
\end{equation}
Furthermore,
\begin{equation}\label{T1 and T2}
T_i^{\varepsilon} = \frac{l_i}{\varepsilon} + O(\ln \varepsilon), \qquad i = 1,2,
\end{equation}
where $l_1$ and $l_2$ are defined in \eqref{li}.
\end{proposition}
\begin{proof}
We provide only an outline of the proof; full details can be found in the work of Carr--Gurtin--Slemrod \cite{CGS1984}.
Applying \eqref{Inverse function-1} with $\mathbf{k}_\delta$ given by Proposition 3.1, we conclude that system \eqref{ODE under scaling transformation} admits at least one two-smooth-interface phase transition solution of single-valley or single-peak type, denoted by $v^{\varepsilon}$, corresponding to the parameters $(\lambda, \sigma)$ defined in \eqref{Inverse function}. 
From Lemma \ref{the properties of I} and \eqref{J01}, it follows directly that
\begin{equation*}
  T_1^{\varepsilon}=J_{01}=\frac{2}{\sqrt{2W''(z_1^{\sigma,\lambda})}}\Big(\ln \big(\eta+\varsigma+\big(\eta^2+2\varsigma\eta\big)^{\frac12}\big)-\ln\varsigma\Big)+g(\sigma,\lambda,\eta),
\end{equation*}
where $\varsigma=\frac{W'(v)-W'(\alpha_\sigma)}{W''(v)}$ and $g(\sigma,\lambda,\eta)$ continuous on $\Omega_1\cup\Omega_0\times[0,\eta]$.
Combining this identity with \eqref{transformation of h} and \eqref{Inverse function-1} yields \eqref{T1 and T2}.
\end{proof}

\

Finally, we will demonstrate that the phase transition solution with two interfaces corresponds to the global minimizer of the minimum value of the variational problem as following:
\begin{equation}\label{Energy}
E^{\varepsilon}(v)=\varepsilon\int_{-\frac1\varepsilon}^{\frac1\varepsilon}\big(W(v)+\frac{1}{2}v_y^2\big)dy,
\end{equation}
over the following set of functions
\begin{equation}\label{the set of H}
 H^1_{\mathrm{per},\bar v}=\Big\{f(x)\in  H^1_{\mathrm{per}}([-\varepsilon^{-1},\varepsilon^{-1}])\Big|\int_{-\frac{1}{\varepsilon}}^{\frac{1}{\varepsilon}}f(x)dx=\frac{2\bar v}{\varepsilon}\Big\}.
\end{equation}

Similar to the derivation of Lemma \ref{equivalent critical points}, classical variational theory ensures the existence of a smooth minimizer for the functional \eqref{Energy} within the admissible class \eqref{the set of H}; any such minimizer must satisfy the associated Euler-Lagrange equation \eqref{Approximate ns-lagrange} under the constraint \eqref{constraint for v,u}.
It follows that the two-smooth-interface phase-transition solution $ v^{\varepsilon}(y) $ derived in proposition \ref{existence of k1 and k2} corresponds to a local minimum of the energy functional. On both computational (see Hsieh-Wang \cite{HW-1997}, He-Liu-Shi \cite{HLS} and references therein) and intuitive grounds, we further contend that this solution in fact furnishes the global energy minimizer among all admissible variational solutions of \eqref{Energy} within the admissible class \eqref{the set of H}.

This physical insight reveals that, although discontinuous density profiles—such as those in \eqref{Piecewise smooth solution}, are theoretically permissible in the inviscid limit without energetic cost, introducing artificial viscosity within a periodic domain selects a single-peak (or single-valley) interface as the unique, energy-minimizing two-interface smooth phase transition solution. This physically meaningful profile is characterized by a single monotonic transition from one phase to the other and back, thereby providing a spatially regularized counterpart to the underlying discontinuous configurations.

As shown in the preceding derivation, the system \eqref{Approximate ns-lagrange} under the constraint \eqref{constraint for v,u} admits a two-interface phase transition solution, denoted by $v^\varepsilon_{\mathrm{single-valley}}$ (equivalently $v^\varepsilon_{\mathrm{single-peak}}$). We define
\begin{equation}\label{Maxwell energy}
  E^{\varepsilon}_0 = E^{\varepsilon}(v^\varepsilon_{\mathrm{single-valley}}) =E^{\varepsilon}(v^\varepsilon_{\mathrm{single-peak}}),
\end{equation}
and refer to $E_0^{\varepsilon}$ as the energy of phase transition solution with two interfaces. Now we  restrict our attention to phase transition solution with two interfaces,  by \eqref{first integral},
\begin{equation}\label{energy}
  W(v)+\frac12v_y^2=2f_{\sigma,\lambda}(v)-\sigma v+\lambda.
\end{equation}
Making use of the first relation in \eqref{Two-point boundary value condition} to change the integration variable, in conjunction with the third integral constraint given in \eqref{ODE under scaling transformation}, we find that
\begin{equation}\label{Another expression of energy}
 E^{\varepsilon}(v)=2(-\sigma \bar v+\lambda)+2\sqrt{2}\varepsilon\int_{z_1^{\sigma,\lambda}}^{z_2^{\sigma,\lambda}}f^{\frac12}_{\sigma,\lambda}(s)ds, \ \ \forall (\sigma,\lambda)\in D^{\varepsilon}(\sigma,\lambda).
\end{equation}
Consequently, the solutions of \eqref{restriction for admissible pair} correspond to the critical points of $E^{\varepsilon}$. This can be verified by direct calculation, which yields:
\begin{equation}\label{critical point}
  \frac{\partial E^{\varepsilon}}{\partial \sigma}=-2\bar v+\sqrt{2}\varepsilon \int_{z_1^{\sigma,\lambda}}^{z_2^{\sigma,\lambda}}s\big(f_{\sigma,\lambda}(s)\big)^{-\frac12}ds,\qquad  \frac{\partial E^{\varepsilon}}{\partial \lambda}=2-2\sqrt{2}\varepsilon \int_{z_1^{\sigma,\lambda}}^{z_2^{\sigma,\lambda}}\big(f_{\sigma,\lambda}(s)\big)^{-\frac12}ds.
\end{equation}

The asymptotic form of the energy for the phase transition solution with two interfaces then follows by applying proposition \ref{prop:two-interface-phase-transition} together with the monotonicity analysis of Carr-Gurtin-Slemrod \cite{CGS1984}. We state this result in the following lemma, omitting the proof to avoid redundancy.
\begin{lemma}
As $\varepsilon \to 0$, the energy $E^{\varepsilon}_0$ of the phase transition solution with two interfaces possesses the asymptotic form
\begin{equation}\label{asymptotic form}
  E^{\varepsilon}_0 = 2(-\sigma_0\bar v + \lambda_0) + \varepsilon S + O(e^{-C/\varepsilon}),
\end{equation}
uniformly for $\bar v$ in any closed subinterval of $(\alpha_0, \beta_0)$, with the constant $S$ given by
\begin{equation}\label{S}
  S = 2\int_{\alpha_0}^{\beta_0}\big[W(s)-W(\beta_0)+\sigma_0(s-\beta_0)\big]^{\frac12}\,ds.
\end{equation}
\end{lemma}

Solutions of \eqref{G} must be of one of the following types: a two-smooth-interface phase transition solution, $2N$-smooth-interface phase transition solutions with $N \geq 2$, or a constant solution.  An analysis of the second variation of the functional $E^{\varepsilon}$ shows that the two-smooth-interface solution of \eqref{G} has the least energy among the latter two types.  We formalize this conclusion in the following lemma.

\ 

\begin{lemma}
Let $\bar{v}$ belong to a closed subset of $\Omega_{\mathrm{Maxwell}}=(\alpha_0, \beta_0)$ and let $\varepsilon > 0$ be sufficiently small. Then \textbf{two-smooth-interface phase transition solution} for $(\varepsilon, \bar{v})$ minimizes the energy over the following two classes:
\begin{enumerate}
    \item[(i)] the set of all constant solutions to \eqref{ODE under scaling transformation};
    \item[(ii)] the set of all $2N$-smooth-interface phase transition solutions with $N \geq 2$ to \eqref{ODE under scaling transformation}.
\end{enumerate}
\end{lemma}

\begin{proof}
To show (i), consider any constant solution $v$ of \eqref{ODE under scaling transformation}. The constraint \eqref{ODE under scaling transformation}$_3$ implies $v \equiv \bar v$, and the corresponding energy obtained from \eqref{Energy} is
\begin{equation}
E^{\varepsilon}(\bar v) = 2W(\bar v).\label{eq:constant_energy}
\end{equation}
Using the asymptotic expression for $E_0^{\varepsilon}$ in \eqref{asymptotic form}, we find
\begin{equation}\label{error estimate}
E^{\varepsilon}(\bar v) - E_0^{\varepsilon} = 2\left[ \left(W(\bar v)+\sigma_0\bar v\right) - \left(W(\beta_0)+\sigma_0 \beta_0\right) \right] + O(\varepsilon),
\end{equation}
uniformly for $\bar v \in [\alpha_0+\delta, \beta_0-\delta]$. The function $W(v) - \sigma_0 v$ has strict global minima at $v = \alpha_0$ and $v = \beta_0$. As $\bar v$ is confined to a closed interval excluding these points, the bracket in \eqref{error estimate} is positive and uniformly bounded away from zero. Hence, for sufficiently small $\varepsilon$, the energy difference is positive, establishing (i).

To prove (ii), let $v$ be a $2N$-smooth-interface phase transition solutions with $N \geq 2$  transitions.  We  show that $v$  cannot be a local minimizer by constructing a variation $\eta\in H^1_{per,0}(-\varepsilon^{-1},\varepsilon^{-1})$ such that the second variation satisfies
\begin{equation}\label{Second variation}
 \varepsilon J(v,\eta)\xlongequal{\mathrm{def}}\frac{d^{2}}{d\xi^2}E^{\varepsilon}(v+\xi\eta)\Big|_{\xi=0}<0, 
\end{equation}To this end, we start from the energy functional \eqref{Energy}, from which the second variation takes the form
\begin{equation}\label{Second variation-1}
 J(v,\eta)=\int_{-\varepsilon^{-1}}^{\varepsilon^{-1}}\Big[\eta_y^2+W''(v)\eta^2\Big]dy.
\end{equation}
Since $v$ is $2N$-smooth-interface phase transition solutions with $N \geq 2$, it exhibits at least four oscillations (i.e., a bimodal or bisingle-valley structure). Thus, there exist points $z\in(-\frac{1}{\varepsilon},\frac1\varepsilon)$ such that
\begin{equation}\label{properties of v}
  v(-\varepsilon^{-1})=v(z),\quad v'(z)=0=v'(-\varepsilon^{-1}).
\end{equation}
We now construct a function $\eta$ as follows
\begin{equation}\label{eta0}
 \eta_0(y)=\left\{\begin{array}{llll}
 v'(y),& -\varepsilon^{-1}\leq y\leq z,\\
 0,& z\leq y\leq \varepsilon^{-1},\end{array}\right.
\qquad
 \eta_1(y)=\left\{\begin{array}{llll}
 1,& y=-\varepsilon^{-1},\\
  0,&z\leq y<\varepsilon^{-1}, \end{array}\right. 
\end{equation}
where $\eta_1\in H^1(-\varepsilon^{-1},\varepsilon^{-1})$ and 
satisfies $\displaystyle\int_{-\varepsilon^{-1}}^{\varepsilon^{-1}}\eta_1dy=0$. Then define
\begin{equation}\label{eta}
  \eta(y)=\eta_0(y)+\vartheta\eta_1(y),
\end{equation}
with $\vartheta$ to be determined. A direct computation gives
\begin{equation}\label{J}
\left.\begin{array}{llll}
\displaystyle J(v,\eta)&=\displaystyle \int_{-\varepsilon^{-1}}^{z}\Big(v''^2+W''(v)v'^2\Big)dy+2\vartheta\int_{-\varepsilon^{-1}}^{z}\Big(v''\eta'_1+W''(v)v'\eta_1\Big)dy+O(\vartheta^2),
\end{array}\right.\end{equation}
From \eqref{ODE under scaling transformation}, we have
\begin{equation}\label{the third derivative of v}
  v'''=W''(v)v',
\end{equation}
Substituting \eqref{the third derivative of v} into \eqref{J}, we obtain
\begin{equation}\label{J-1}
  J(v,\eta)=-4\vartheta v''(-\varepsilon^{-1})+O(\vartheta^2).
\end{equation}
Since $v$ is non-constant, so $v''(-\varepsilon^{-1})\neq0$. Therefore, by choosing 
 $\vartheta$ appropriately, we can ensure  $J(v,\eta) < 0$, which completes the proof.
\end{proof} 

\

It follows that, among all solutions of system \eqref{Approximate ns-lagrange}, the two-smooth-interface phase transition solution (as defined in \ref{two-smooth-interface-phase-transition-solution}) minimizes the energy, a result we summarize in the following proposition.

\begin{proposition}\label{unique}
Under the assumptions of proposition \ref{prop:two-interface-phase-transition}, \textbf{the two-smooth-interface phase transition solution} for $(\varepsilon, \bar v)$ globally minimizes the energy when $\bar v$ lies in any closed subset of the Maxwell region $\Omega_\text{Maxwell}=(\alpha_0,\beta_0)$ and $\varepsilon$ is sufficiently small.
\end{proposition}

The preceding lemmas and the fact that our solution is a two-smooth-interface phase transition allow the results of Carr-Gurtin-Slemrod \cite{CGS1984}, obtained for non‑constant, monotone increasing $v'(y)\ge 0$, to carry over directly. Therefore, the quantities $|z_1^{\sigma,\lambda}-\alpha_0|$ and $|z_2^{\sigma,\lambda}-\beta_0|$ enjoy the asymptotic properties stated in Proposition \ref{prop:two-interface-phase-transition}. We outline the argument below for completeness and brevity. We begin by examining the critical points of $W(v) - \sigma v$, whose behavior tend toward the two endpoints of the Maxwell region in the limit $\varepsilon \to 0$. Lemmas \ref{lem:7}--\ref{nr-property} are quoted directly from \cite{CGS1984}; here we state them without proof.
\section*{Funding} 
Yazhou Chen acknowledges support from National Natural Sciences Foundation of China (NSFC) (No. 12471207),(No. 12371434).\\
Qiaolin He acknowledges support from  National Natural Science Foundation of China (Grant No. 12371434) and the National Key R \& D Program of China Under Grant (No.2022YFE03040002)\\
Dongjuan Niu acknowledges support from the Natural Science Foundation of Beijing Municipality, China (Grant No. 1252004) and Tianyuan Fund for Mathematics of the National Natural Science Foundation of China (Grant No. 12526429).\\
Yi Peng acknowledges support from National Natural Sciences Foundation of China (NSFC) (No. 12301266).\\
Xiaoding Shi acknowledges support from National Natural Sciences Foundation of China (NSFC) (No. 12171024), Beijing Natural Science Foundation  (No. 1252004).\\

\section*{Ethical approval}
Not applicable

\section*{Data Availability Statement} 
The data that support the findings of this study are available from the corresponding author
upon reasonable request.

\section*{Conflict of Interest}
 The authors declare that they have no conflict of interest.

\section*{Clinical Trial Number in the manuscript}
Clinical trial number: not applicable.

\end{document}